\documentclass[11pt,letterpaper]{amsart}
\usepackage[a4paper, total={6in, 9in}]{geometry}

\usepackage[english]{babel}
\usepackage[utf8]{inputenc}
\usepackage{amsmath}
\usepackage{graphicx}
\usepackage{amssymb}
\usepackage{stmaryrd}
\usepackage{amsthm}
\usepackage{tikz-cd}
\usepackage{mathrsfs}
\usepackage{csquotes}
\usepackage[colorinlistoftodos]{todonotes}
\usepackage{yfonts}
\usepackage{ dsfont }
\usepackage{xcolor}
\usepackage{multicol}
\usepackage{indentfirst}
\usepackage[hyphens]{url}

 \usepackage[
    backend=biber,
    style=alphabetic, maxbibnames=99, url=false
  ]{biblatex}
\usepackage{hyperref}
\usepackage{xurl}
\hypersetup{breaklinks=true,
    colorlinks=true,
    linkcolor=magenta,
    citecolor=cyan,
    filecolor=magenta,      
    urlcolor=purple,
    pdftitle={Overleaf Example},
    pdfpagemode=FullScreen,
    }

\allowdisplaybreaks

\usepackage[inline]{enumitem}   
\makeatletter
\newcommand{\inlineitem}[1][]{%
\ifnum\enit@type=\tw@
    {\descriptionlabel{#1}}
  \hspace{\labelsep}%
\else
  \ifnum\enit@type=\z@
       \refstepcounter{\@listctr}\fi
    \quad\@itemlabel\hspace{\labelsep}%
\fi}
\makeatother
\usepackage[capitalise]{cleveref}

\title{Complete Subvarieties of $\rm{M}_{g,n}$ and a Lifting Problem}

\author{Daebeom Choi}
\address{Department of Mathematics\\
    University of Pennsylvania\\
    Philadelphia, PA 19104-6395}
\email{dbchoi@sas.upenn.edu}

\date{\today}
\theoremstyle{definition}
\newtheorem{thm}{Theorem}[section]
\newtheorem{Pthm}{Prior Results}[section]
\newtheorem{lem}[thm]{Lemma}
\newtheorem{prob}[thm]{Problem}

\newtheorem{prop}[thm]{Proposition}

\newtheorem{Kthm}[thm]{Kodaira-Parshin}

\newtheorem{defn}[thm]{Definition}

\newtheorem{conj}[thm]{Conjecture}

\newtheorem*{conj*}{The Lifting Conjectures}

\newtheorem{cor}[thm]{Corollary}

\newtheorem{rmk}[thm]{Remark}

\begin{document}
\maketitle

\begin{abstract}
    Finding the maximal dimension of complete subvarieties of the moduli space of smooth $n$-pointed curves of genus $g$ is a long-standing open problem. Here we show that for $g\ge 3\cdot 2^{d-1}$, if the characteristic of the base field is greater than $2$, then $\rm{M}_g$ contains a complete subvariety of dimension $d$. Furthermore, in positive characteristic, we construct a complete surface in $\rm{M}_{g,n}$ for $g\ge 3$ and $n\ge 1$, which contain a general point. These results follow from the proofs of the lifting conjectures, introduced here. In particular, we translate the existence of complete subvarieties to properties of line bundles on $\rm{M}_{g,n}$. Our method reframes Zaal's approach, with increased efficiency via Keel's results on semi-ample line bundles in positive characteristic. This method demonstrates the difference in the geometry of moduli spaces between characteristic $0$ and characteristic $p$. 
\end{abstract}

\section{Introduction}
The Deligne-Mumford stack $\mathcal{M}_{g,n}$ parameterizes families of smooth curves of genus $g$ with $n$ ordered, distinct marked points over an algebraically closed field $K$. If $2g-2+n>0$,  this is (coarsely) represented by the moduli space ${\rm{M}}_{g,n}$. As has been established, the very nature of the moduli space of curves as a variety has an impact on what we know about families of smooth genus $g$ curves (e.g.~\cite{EH87, Far00, FJP20, HM82}). Very basic questions remain open about the geometry of the moduli spaces. For instance, while $\rm{M}_2$ is an affine $3$-fold, and more generally $\rm{M}_{g,n}$ is a quasi-projective variety of dimension $3g-3+n$, we don't know ``how close" $\rm{M}_{g,n}$ is to being affine or projective.

 One way to evaluate this question is to determine the dimension $r(g,n)$ of the largest complete subvariety of $\operatorname{M}_{g,n}$. For convenience, let $r(g):=r(g,0)$. Finding $r(g,n)$ is a problem going back at least to the 1974 paper by Oort \cite{Oor74}. A number of results are known, as outlined in \cite{Har84, Oor95, FL99} (and briefly in \cref{Pthm:results}). Diaz's upper bound $r(g)\le g-2$ in \cite{Di84} and Zaal's lower bound $r(g)>\log_2 g -2$ in \cite{Za05} are among them. Diaz's result has motivated the conjecture that $r(g)=g-2$, although a determination of the exact values of $r(g,n)$ (or even $r(g)$) remains an open problem.

Our first result improves Zaal's bound for positive characteristic (see \cref{cor:conjpos}):

\begin{thm}\label{thm:mainthm1} The moduli space $\rm{M}_g$ of smooth curves of genus $g$ contains a complete subvariety of dimension $d$, for $g\ge 3\cdot 2^{d-1}$ if $\text{char }K>2$, and for $g\ge \frac{5\cdot 3^{d-1}+1}{2}$ if $\text{char }K=2$.
\end{thm}

One may also ask for the maximal integer $\Tilde{r}(g,n)$ for which a general point in $\rm{M}_{g,n}$ is contained in a $\Tilde{r}(g,n)$-dimensional complete subvariety (setting $\Tilde{r}(g):=\Tilde{r}(g,0)$). This is perhaps a more interesting problem than the former (see \cite[p. 55]{HM98} and \cite[Section 11]{FL99}). Our second theorem is the first non-trivial result about $\Tilde{r}(g,n)$ for $n\ge 3$ (see \cref{cor:conj2pos}).

\begin{thm}\label{thm:mainthm2} If $\text{char }K\ne 0$, $g\ge 3$, and $n\ge 1$, then any point of $\rm{M}_{g,n}$ is contained in a $2$-dimensional complete subvariety. 
\end{thm}

To motivate our methods, following  \cite{FL99, Har84, Har85, HM98, Oor74, Oor95, Za05}, we briefly review  prior results about $\Tilde{r}(g,n)$ and $r(g,n)$ and compare them to our findings.

\begin{Pthm}\label{Pthm:results}
\begin{enumerate}
    \item $\Tilde{r}(g)\ge 1$ if $g\ge 3$. \inlineitem  $r(g)\le g-2$.
    \item $r(g)\ge d$ holds for $g\ge 2^{d+1}$ if $\text{char }K\ne 2$, $g\ge \frac{3^{d+1}+1}{2}$ if $\text{char }K=2$. Furthermore, if $\text{char }K>2$, then $r(6)\ge 2$.
\end{enumerate}
\end{Pthm} 

Part (1) follows from the fact that the boundary of the Satake compactification of $\rm{M}_g$ has codimension $2$ (see \cite{Oor95}). Indeed, in this case, any point in $\rm{M}_g$ is contained in a complete subcurve. Diaz proved (2) in characteristic $0$ \cite{Di84} and Looijenga proved it in general characteristic  \cite{Lo95}. Motivated by this theorem, it is conjectured that $r(g)=g-2$. However, as stated above, this conjecture is open. For instance, it is still unknown whether $r(4)=2$. Zaal proved (3) in his thesis \cite{Za05}, and the second assertion is also published in \cite{Za99}. While he did not prove the characteristic $2$ part of the first assertion of (3), using the method from \cite[Chapter 2]{Za05} with triple covers, instead of double covers, the statement follows (cf. \cite[Fact 1.12]{Za05}).

The Kodaira-Parshin construction (KPC for short) \cite{Par68, Ma85} is behind almost every assemblage of complete subvarieties of $\rm{M}_g$, including Zaal's work above. KPC relates the problem of finding complete subvarieties of $\rm{M}_g$ to the same problem on $\rm{M}_{g,n}$. 

\begin{Kthm}[Heuristic Version] \label{Kthm:kodairashort}
    Given a family of $n$-pointed smooth curves of genus $g$, by taking degree $d$ covers of each fiber, ramified exactly at the marked points, we obtain a family of smooth curves of genus $g'=g'(g,n,d)$.
\end{Kthm}

For a rigorous formulation of a variant of the KPC that used in this paper, see \cref{Kthm:kodaira}. To apply the KPC, one needs to start with a complete family of pointed curves $S\to \mathcal{M}_{g,n}$, whose existence is already a very hard problem. To deal with this, in the proof of the first assertion of \cref{Pthm:results} (3), in \cite[Lemma 2.2]{Za05}, Zaal used an efficient form of the KPC, which results in an iterative production of complete subvarieties of $\rm{M}_{g,2}$, giving the best known lower bound for $r(g)$. We have found this method difficult to improve, since we need to fix a base curve $X_0$ in \cite[Lemma 2.2]{Za05}.

In the proof of the second assertion of \cref{Pthm:results} (3),  \cite{Za99}, Zaal used a new method of constructing complete subvarieties of $\rm{M}_{g,n}$ based on Keel's semiampleness criteria in positive characteristic \cite{Ke99}. Applying the KPC with this proves the theorem.

Following \cite[Section 11]{FL99}, we briefly summarize what is known in case $n\ge 1$:

\begin{Pthm}\label{Pthm:Results2}
    \begin{enumerate}
        \item $r(g,1)=r(g)+1$, $r(g,n+1)\le r(g,n)$ (and same for $\Tilde{r}$);
        \item $r(g,n)\ge 1$ for $g\ge 2$; \inlineitem $r(g,2)\ge 2$, for $\text{char }K\ne 0$, $g\ge 3$; and
        \inlineitem $\Tilde{r}(g,2)\ge 1$.
    \end{enumerate}
\end{Pthm}

Since $\mathcal{M}_{g,n+1}\to \mathcal{M}_{g,n}$ is an affine morphism if $n\ge 1$ and a projective morphism if $n=0$, of relative dimension $1$, assertion (1) is straightforward. To prove (2), one can use a construction involving elliptic curves (for example, see \cite[Theorem 2.3]{Za05}). Similarly, (4) follows from a construction using Jacobian: see \cite[Section 11]{FL99}. (3) is proved in \cite{Za99}, as a step in the proof of part $(3)$ of \cref{Pthm:results}. Note that \cite[Lemma 2.2, Theorem 2.3]{Za05} also provides a way to construct many complete subvarieties of $\rm{M}_{g,n}$.

Now we will explain our approach to this problem. Based on the previous discussion, our main theorems can be reduced to the problem of finding many complete subvarieties of $\rm{M}_{g,n}$. Note that \cref{Pthm:Results2} (1) is proven by `descending' complete subvarieties of $\rm{M}_{g,n+1}$ to $\rm{M}_{g,n}$. Instead, we do exactly the opposite, by lifting (complete) subvarieties of $\rm{M}_{g,n}$ to $\rm{M}_{g,n+1}$. This approach allows us to systematically study (complete) subvarieties of $\rm{M}_{g,n}$ and construct a large number of such subvarieties that are sufficient to prove \cref{thm:mainthm2}. Additionally, it enables us to more efficiently apply the KPC and ultimately produce the improved lower bound in positive characteristic as stated in \cref{thm:mainthm1}. A rigorous statement of the lifting problem is formulated in the following conjectures. Such a lifting problem was implicitly considered in \cite{Za99} using \cite{Ke99}. From now on, $\pi:\mathcal{M}_{g,n+1}\to \mathcal{M}_{g,n}$ will always denote the map dropping $(n+1)$st point.

\begin{conj*}\label{conj:mconjshort}

    Assume $2g-2+n>0$ and $n\ge 1$. Let $S$ be a variety, $u:S\to \mathcal{M}_{g,n}$ a morphism, and $x\in \mathcal{M}_{g,n+1}$ be a closed point such that $\pi(x)\in u(S)$.
\begin{enumerate}
    \item{Weak Version, \cref{conj:mconj} :} There exists a proper surjective morphism $p: S'\to S$, and a morphism $u':S'\to \mathcal{M}_{g,n+1}$ such that the following diagram ($2$-)commutes.
    \[\begin{tikzcd}
    S' \arrow[r, "u'"]\arrow[d, "p"]& \mathcal{M}_{g,n+1}\arrow[d, "\pi"]\\
    S  \arrow[r, "u"]& \mathcal{M}_{g,n}
\end{tikzcd}\]
    \item {Strong Version, \cref{conj:mconj2} :} We can find such $S',p, u'$ that satisfies $x\in u'(S')$.  
\end{enumerate}
\end{conj*}

 These conjectures are discussed in more detail in \cref{sec:KP}.  In particular, they imply the existence of many complete subvarieties in $\rm{M}_{g,n}$. Our main theorem is the following.

\begin{thm}\label{thm:themain}
    The Lifting Conjectures hold in positive characteristic. The Lifting Conjectures (1) fail in characteristic $0$ when $g\ge 3$ and $n>2g+2$.
\end{thm}

We state one important corollary of our main theorem.

\begin{cor}\label{cor:reduce}
    If $\text{char }K>0$, then $r(g,n)=r(g)+1$ and $\Tilde{r}(g,n)=\Tilde{r}(g)+1$ for $n\ge 1$.  
\end{cor}

\cref{cor:reduce} follows from \cref{thm:themain} (see \cref{thm:mapp} and \cref{thm:mapp2}), and implies \cref{thm:mainthm1} and \cref{thm:mainthm2}. Note that \cref{cor:reduce} reduces the calculation of $r(g,n)$ and $\Tilde{r}(g,n)$ to that of $r(g)$ and $\Tilde{r}(g)$ in positive characteristic.

We will now introduce our method. To prove \cref{thm:themain}, we use what we call the method of universal DNS line bundles, cf. \cref{defn:univdns}. This method is a refinement of the approach used in \cite{Za99}, where Zaal employed a line bundle and a morphism associated with it to produce a $2$-dimensional complete subvariety of $\rm{M}_{g,2}$. In \cref{defn:dnsline}, we introduce a generalization of this, which we call a line bundle that ``detects a new section" (DNS line bundle for short), whose existence is essentially equivalent to \cref{conj:mconj}. We also introduce a universal DNS line bundle in \cref{defn:univdns}, which allows us to convert questions about complete subvarieties to well-studied problems about certain line bundles on the moduli space of curves. Moreover, using universal DNS line bundles and \cref{conj:mconj}, we can prove \cref{conj:mconj2} with relative ease, even though the latter is much stronger. This further illustrates the advantage of using the method of universal DNS line bundles.

While this approach is not strictly necessary for the proof of \cref{conj:mconj} (see \cref{rmk:dnsuse}), it provides a concise proof of \cref{conj:mconj2}. Additionally, this method plays an essential role in constructing the counterexample in characteristic zero, as stated in \cref{thm:themain}. By classifying all possible universal DNS line bundles, we obtain a complete understanding of the counterexample.

Additionally, we explain the significance of the transformation of problems from complete subvarieties to line bundles on the moduli space of curves mentioned previously. The moduli space of curves has a rich structure that provides an array of techniques to establish several valuable theorems on line bundles. For instance, Keel used the reduction through clutching maps between moduli spaces of curves to prove \cref{thm:omegasemi}, which is an essential element in the proof of the Lifting conjectures for positive characteristics. Furthermore, we propose a problem, \cref{conj:onlyhope}, that is solely a statement concerning a line bundle on the moduli space of curves, but its solution implies a significant theorem on the complete subvariety problem. Additionally, this `universal' approach presents a way to investigate the Lifting conjectures even when no universal DNS line bundles exist, as seen in \cref{prob:level}.

The counterexample to the Lifting Conjectures in characteristic $0$ when $n$ is sufficiently large reveals a significant difference between the geometry of $\mathcal{M}_{g,n}$ in characteristic $0$ and characteristic $p$. This echoes similar findings about the analogous problem for moduli of abelian varieties found by Keel and Sadun \cite{KS03} (see \cref{rmk:KeelSadun} for further discussion). It also raises the natural question about what values of $n$  the Lifting Conjecture holds. These issues are discussed in more detail in \S \ref{sec:char0}.

\textbf{Plan of the paper}  In \cref{sec:KP}, we discuss, provide  equivalent formulations, and applications of \cref{conj:mconj} and \cref{conj:mconj2}. In \cref{sec:LNS} we describe DNS line bundles, and their properties. In \cref{sec:results}, we prove \cref{thm:themain} and discuss differences between our findings in characteristic $p$ and characteristic $0$. We also present \cref{conj:onlyhope}, an assertion about a particular line bundle on the moduli space of curves, that if true,  would imply nontrivial results about complete subvarieties in characteristic $0$.

\textbf{Notation} Let  $\overline{\mathcal{M}}_{g,n}$ denote the moduli stack of stable $n$-pointed curves of genus $g$, and if $2g-2+n>0$, let $\overline{\operatorname{M}}_{g,n}$ be the corresponding coarse moduli spaces. By $\pi:\mathcal{U}_{g,n}\to \mathcal{M}_{g,n}$ and $\pi:\overline{\mathcal{U}}_{g,n}\to \overline{\mathcal{M}}_{g,n}$ we mean the universal curves on $\mathcal{M}_{g,n}$ and $\overline{\mathcal{M}}_{g,n}$ with universal sections $\bar{s}_1,\cdots, \bar{s}_n$ and let $\overline{S}_1,\cdots, \overline{S}_n$ be their images. Let $\pi_i:\overline{\mathcal{M}}_{g,n+1}\to \overline{\mathcal{M}}_{g,n}$ the map forgetting the $i$th point. Our notation is different from \cite{Ke99}: our $\overline{\mathcal{U}}_{g,n}$ is same as $\mathcal{U}_{g,n}$ of \cite{Ke99}. Moreover, let $\partial \mathcal{M}_{g,n}$ be the locus of singular curves. By this boundary of the moduli space of curves we mean the `Deligne-Mumford boundary' of \cite{Ke99}. Note that $\overline{\mathcal{U}}_{g,n}$ is canonically isomorphic to $\overline{\mathcal{M}}_{g,n+1}$. In this paper, we will implicitly identify them.

Here $f: C\to S$ will denote a family of stable curves over the base scheme $S$, and $g$ be the genus of fiber. The scheme $S$ is a $K$-scheme, and, in many cases, $S$ will be a variety. In this paper, we assume that all varieties are irreducible. For $f: C\to S$ a family of curves with $n$ marked points, we will denote the $n$ disjoint sections of $f$ by $s_1,\cdots, s_n: S\to C$. Here, $\Omega_{C/S}^1$ means the relative dualizing sheaf (not the sheaf of relative Kahler differentials). For $1\le i\le n$, let $S_i$ be the image of $s_i$. Since $s_i$ are sections of a proper map of relative dimension 1 and their image is a smooth point, $S_i$ are Cartier divisors.

\section*{Acknowledgement}

The author would like to thank Angela Gibney for introducing the problem, her continued support, and helpful discussions. The author would also like to thank David Harbater for providing a very interesting example, \cref{rmk:counter2}, related to the lifting conjectures.

\section{the lifting conjectures, equivalent formulations  and applications}\label{sec:KP}

The following is a rigorous formulation of (a variant of) \cref{Kthm:kodairashort}.

\begin{Kthm}\label{Kthm:kodaira}\cite{Par68, Har85, Ma85, Za05}
    Assume $\text{char }K\ne 2$. Let $f:C\to S$ be a family of genus $g$ curves with $2n$ marked points $s_1,\cdots, s_{2n}:S\to C$ over a variety $S$. Then there exist a finite map $p:S'\to S$ and a family of genus $2g+n-1$ curves $f':C'\to S'$ such that for any closed point $x\in S'$, $f'^{-1}(x)$ is a double cover of $f^{-1}(p(x))$ ramified exactly at $s_1(p(x)),\cdots, s_{2k}(p(x))$. If $\text{char }K\ne 3$ and $f: C\to S$ be a family of genus $g$ curves with $n$ marked points, then the same construction with triple covers gives a family of genus $3g+n-2$ curves over $S'$.
\end{Kthm}

Motivated by \cref{Kthm:kodaira}, we proposed \cref{conj:mconj} and \cref{conj:mconj2}, stated in the introduction as the Lifting Conjectures, with the aim of finding complete subvarieties of $\operatorname{M}_{g,n}$.  Here in \cref{sec:ConRemarks} we discuss these conjectures, giving  reformulations in \cref{sec:equivchar}, and applications in \cref{sec:App1Proof}. 

\subsection{The lifting conjectures}\label{sec:ConRemarks} 
    Here, we state a weak version of our lifting conjecture.

    \begin{conj}\label{conj:mconj}
Assume $2g-2+n>0$ and $n\ge 1$. Let $S$ be a variety, and $u:S\to \mathcal{M}_{g,n}$ a morphism. Then, there exists a proper surjective morphism $p: S'\to S$, and a morphism $u':S'\to \mathcal{M}_{g,n+1}$ such that the following diagram ($2$-)commutes.
\[\begin{tikzcd}
    S' \arrow[r, "u'"]\arrow[d, "p"]& \mathcal{M}_{g,n+1}\arrow[d, "\pi"]\\
    S  \arrow[r, "u"]& \mathcal{M}_{g,n}
\end{tikzcd}\]
\end{conj}

    \cref{conj:mconj} is stated with the language of stacks because as it is more comprehensive and useful in some cases (cf. \cref{rmk:counter}, \cref{rmk:stack}). Except for this section, we will consider this conjecture in the form of (3) or (4) in \cref{prop:equiv}.

    \begin{rmk}\label{rmk:counter}
         In \cref{conj:mconj} we require $n\ge 1$ as the statement is trivial when $n=0$. However, the condition $2g-2+n>0$ is essential. Consider the case $(g,n)=(0,2)$. In this case, $\mathcal{M}_{0,2}=\text{B}\mathbb{G}_m$, $\mathcal{M}_{0,3}=\text{Spec }K$, and $\mathcal{M}_{0,3}\to \mathcal{M}_{0,2}$ is the natural projection map $\text{Spec }K\to \text{B}\mathbb{G}_m$. A morphism $S\to \text{B}\mathbb{G}_m$ corresponds to a line bundle on $S$, and this morphism lifts along $\text{Spec }K\to \text{B}\mathbb{G}_m$ if and only if the line bundle is trivial. Therefore, \cref{conj:mconj} for $(g,n)=(0,2)$ is equivalent to the following: for any line bundle $\mathcal{L}$ on $S$, there exists a proper surjective map $p: S'\to S$ such that $p^\ast \mathcal{L}$ is trivial. However, if $S$ is a projective variety, and $\mathcal{L}$ is an ample line bundle, then $p^\ast \mathcal{L}$ cannot be trivial. Therefore, \cref{conj:mconj} is not true in this case. Note that if $S$ is projective and smooth, then by \cite[Theorem 1.1]{BS11}, such $p$ exists if and only if $\mathcal{L}$ is essentially finite as a vector bundle. 
    \end{rmk}

    \begin{rmk}\label{rmk:counter2}
        It is essential to allow a proper (or finite, see \cref{prop:equiv}) base change. Consider the case where $S$ is a proper smooth curve of genus at least $2$ and $C\to S$ is a non-isotrivial family of curves of genus $g\ge 2$ over $S$. Let $s_1,\cdots,s_n: S\to C$ be a collection of pairwise disjoint sections. Then, the intersection number of $s_i$ and $s_j$ is zero if $i\ne j$. By basic intersection theory, the self-intersection number of $s_i$ is $-\deg s_i^\ast\Omega_{C/S}^1$. By \cite[Proposition 6]{Par68}, $s_i^\ast\Omega_{C/S}^1$ is ample, and hence the self-intersection number is negative. Therefore, $s_i$ are linearly independent vectors in $\text{NS}(C)\otimes \mathbb{Q}$. We conclude that $n$ must be smaller than or equal to the Picard number of $C$. In particular, we cannot recursively find a new disjoint section without a finite base change. In \cite[Chapter 5]{Za05}, Zaal explicitly constructs a family of genus $31$ smooth curves without any section.

        Another very interesting example occurs in \cite{CH85}. Let $\left(\mathbb{P}^1\right)^n$ be a direct product of projective lines and let $\Delta_n\subseteq \left(\mathbb{P}^1\right)^n$ be the closed subscheme consisting of points that have two equal coordinates. Then there exists a natural map $\left(\mathbb{P}^1\right)^n\setminus \Delta_n\to \mathcal{M}_{0,n}$, which is a surjective and smooth map whose fiber is $\left(\mathbb{P}^1\right)^3\setminus \Delta_3$. The pullback of this map along $\mathcal{M}_{0,n+1}\to \mathcal{M}_{0,n}$ is the projection map $\pi_n:\left(\mathbb{P}^1\right)^{n+1}\setminus \Delta_{n+1}\to \left(\mathbb{P}^1\right)^n\setminus \Delta_n$. \cref{conj:mconj} holds if and only if a section of this projection map exists after a proper (or finite) base change. We will prove this in \cref{thm:kapapp}. Moreover, if our base field is $\mathbb{C}$, according to \cite[Lemma 1.1]{CH85}, there exists a continuous section of $\pi_n$. However, as Coombes and Harbater showed in the remark after \cite[Lemma 1.1]{CH85}, there is no (algebraic) section of $\pi_n$. This also reveals why we need proper/finite base change in \cref{conj:mconj}.
    \end{rmk}

Some authors already implicitly used a small part of this conjecture to construct complete subvarieties of $\rm{M}_g$. For example, in \cite{Za99}, Zaal constructed a complete subsurface of $\rm{M}_6$ essentially by proving the following special case of \cref{conj:mconj}.

\begin{Pthm}\label{Pthm:Zathm}\cite{Za99}
    Let $f: C\to S$ be a complete non-isotrivial family of smooth curves of genus $g$ over a complete curve $S$. This induces a morphism $u:C\to \mathcal{M}_{g,1}=\mathcal{U}_{g}$. Then $u$ satisfies \cref{conj:mconj}.
\end{Pthm}

To calculate $\Tilde{r}(g,n)$, we need a stronger version of \cref{conj:mconj}.

\begin{conj}\label{conj:mconj2}
Assume $2g-2+n>0$ and $n\ge 1$. Let $S$ be a variety over the field $K$ and $u:S \to \mathcal{M}_{g,n}$ a morphism. Let $x\in \mathcal{M}_{g,n+1}$ be a closed point such that $\pi(x)\in u(S)$. Then there exist a variety $S'$, a proper surjective morphism $p:S'\to S$, and a morphism $u':S'\to \mathcal{M}_{g,n+1}$ such that the following diagram ($2$-)commutes and $x\in u'(S')$.
\[\begin{tikzcd}
    S' \arrow[r, "u'"]\arrow[d, "p"]& \mathcal{M}_{g,n+1}\arrow[d, "\pi"]\\
    S  \arrow[r, "u"]& \mathcal{M}_{g,n}
\end{tikzcd}\]
\end{conj}

To the author's knowledge, there are no known special cases of \cref{conj:mconj2}. In \S \ref{sec:charp}, we will prove both \cref{conj:mconj} and \cref{conj:mconj2} is true in characteristic $p$.

Note that \cref{conj:mconj} and \cref{conj:mconj2} are flexible and can be modified. For example, proving \cref{conj:mconj} for proper varieties or even for a specific proper subvariety is sufficient to construct complete subvarieties of $\rm{M}_{g,n}$. We will discuss this in \S \ref{sec:char0} in more detail. This becomes important in \S \ref{sec:char0} since \cref{conj:mconj} is false if the characteristic of the base field is zero and $n$ is sufficiently large (see \cref{thm:conjfalse}).

\subsection{Equivalent characterizations of the lifting conjectures} \label{sec:equivchar}

As mentioned above, almost every result regarding complete subvarieties of $\rm{M}_g$ utilizes the Kodaira-Parshin construction, which reformulates the problem to finding complete subvarieties of $\rm{M}_{g,n}$.  In \cref{conj:mconj} a  systematic method for solving the latter is proposed. Here, we prove \cref{prop:equiv} which gives several conditions equivalent to \cref{conj:mconj}.

\begin{prop}\label{prop:equiv}
    Let $f:C\to S$ be a family of curves of genus $g$ with $n$ marked points $s_1,\cdots, s_n:S\to C$ on a variety $S$. Then the following are equivalent.
    \begin{enumerate}
        \item The morphism $u:S\to \mathcal{M}_{g,n}$ corresponding to $f$ satisfies \cref{conj:mconj}
        \item Same as (1), but $p:S'\to S$ is required to be finite.
        \item There is a closed codimension $1$ subscheme $T$ of $C$ that does not intersect with the images of the sections $s_i$ for all $i$. 
        \item There is a proper surjective map $p:S'\to S$ such that after a base change along $S'\to S$, we can find a section disjoint from $s_i$ for every $i$, 
    \end{enumerate}
\end{prop}

\begin{proof}
(1) implies (2): Let $p:S' \to S$ be a proper surjective map that satisfies (1). Then $p$ factors through $S\times_{\mathcal{M}_{g,n}}\mathcal{M}_{g,n+1}$, which is also a variety, and the projection map $\pi_1:S\times_{\mathcal{M}_{g,n}}\mathcal{M}_{g,n+1}\to S$ has affine fibers. Consider the scheme-theoretic image $S''$ of $S'\to S\times_{\mathcal{M}_{g,n}}\mathcal{M}_{g,n+1}$. Then the natural projection map $S'' \to S$ is proper. Since every fiber of $\pi:\mathcal{M}_{g,n+1}\to \mathcal{M}_{g,n}$ is affine, so is $S''\to S$. Since every fiber of $S''\to S$ is proper and contained in an affine scheme, it is finite. Therefore, $S''\to S$ is proper and quasi-finite, hence also finite by Zariski's main theorem. This shows $S''\to S$ is the desired map.

(2) implies (3): Let $f':C'\to S'$ and $s_1',\cdots, s_n':S'\to C'$ be the base change and $s_{n+1}:S'\to C'$ be the new section. Let $r$ be the composition of $s_{n+1}$ and $C'\to C$ and $T$ be the scheme-theoretic image of $r$. Since $p: S'\to S$ is finite, one can see that $T$ is a codimension $1$ closed subscheme of $C$. Moreover, by definition, $T$ does not intersect with $s_1,\cdots, s_n$. Hence $T$ is the desired subscheme.

(3) implies (4): Let $S'=T$. Then the natural projection $S'=T\to S$ is proper. Also, for the same reason as above, $S'\to S$ is quasi-finite, hence also finite. Since $\dim S=\dim S'$, the image of the projection map $S'\to S$ should be a closed subvariety of the same dimension, hence it is surjective by the irreducibility of $S$. The inclusion $S'\to C$ gives a new section after the base change. Therefore, $S'\to S$ satisfies the condition.

(4) is just another description of (1).
\end{proof}

An analogous form of \cref{prop:equiv} holds for \cref{conj:mconj2}.

\begin{prop}\label{prop:equiv2}
    Let $f:C\to S$ be a family of curves of genus $g$ with $n$ marked points $s_1,\cdots, s_n:S\to C$ on a variety $S$. Then the following are equivalent.
    \begin{enumerate}
        \item The morphism $u:S\to \mathcal{M}_{g,n}$ corresponds to $f$ satisfies the \cref{conj:mconj2}
        \item Same as (1), but $p:S'\to S$ is required to be finite.
        \item For any closed point $x\in C$ not contained in $S_1,\cdots, S_n$, there exists a closed codimension $1$ subscheme $T$ of $C$ that does not intersect with $S_i$'s and contains $x$. 
    \end{enumerate}
\end{prop}

The proof is essentially the same as the proof of \cref{prop:equiv}, and we omit it.

Note that \cref{conj:mconj} and \cref{conj:mconj2} are just handy forms of more general, essentially equivalent conjectures. To show this, we will introduce a lemma.

\begin{lem}\label{lem:covmod}
    Assume $2g-2+n>0$.
    \begin{enumerate}
        \item There is a smooth variety $X$ and a finite, surjective, etale morphism $u_X:X\to \mathcal{M}_{g,n}$.
        \item There is a $K$-scheme $\overline{Y}$ and a finite, surjective, flat morphism $u_Y:\overline{Y}\to \overline{\mathcal{M}}_{g,n}$.
    \end{enumerate}
\end{lem}

\begin{proof}
    (1) follows from the construction of moduli space of smooth curves with an abelian level structure. See \cite{Mum83}. (2) is a direct corollary of \cite[Theorem 2.1]{KV04}.
\end{proof}

\begin{prop}\label{prop:general}
    The following are equivalent.
    \begin{enumerate}
        \item There exists a smooth variety $S$ and a finite, surjective, flat morphism $u: S \to \mathcal{M}_{g,n}$ such that \cref{conj:mconj} (resp. \cref{conj:mconj2}) holds for $u$.
        \item \cref{conj:mconj} (resp. \cref{conj:mconj2}) holds for every smooth variety $S$.  
        \item \cref{conj:mconj} (resp. \cref{conj:mconj2}) holds for every $K$-scheme $S$ with a finite, surjective, flat map $p:S'\to S$.
    \end{enumerate}
\end{prop}

\begin{proof}
    We prove this for \cref{conj:mconj} as the argument for \cref{conj:mconj2} is analogous. (3) is strictly stronger than (2). By \cref{lem:covmod} (1), (2) implies (1). Hence it is enough to show that (1) implies (3). Let $u_X: X\to \mathcal{M}_{g,n}$ be a finite flat and surjective map from a smooth variety $X$. Let $f_X: C_X\to X$ and $s_{X,1},\cdots, s_{X,n}: X\to C_X$ be the corresponding family of smooth curves and marked points. By the assumption, $u_X: X\to \mathcal{M}_{g,n}$ satisfies \cref{conj:mconj}, and by \cref{prop:equiv} (3), there exists a closed codimension $1$ subscheme $T$ of $C_X$ that does not intersect with $s_{X,1},\cdots, s_{X,n}$. Since $C_X$ is also a smooth variety and $T$ is an effective divisor of $C_X$, $T\to C_X$ is a regular embedding. Therefore, $T\to X$ is an lci morphism. Using the same argument as \cref{prop:equiv}, we can show that $T\to X$ is also finite and surjective. Since the finite surjective lci morphism is flat \cite[Exercise 3.5, Chapter 6]{Liu02}, $T\to X$ is a finite flat surjective map. 

    \[\begin{tikzcd}
     S' \arrow[rr]\arrow[rd] & & C_0 \arrow[rr]\arrow[rd]\arrow[dd] & & C \arrow[dd]\arrow[rd] & \\
     & T\arrow[rr] & & C_X\arrow[dd] \arrow[rr] & & \mathcal{U}_{g,n}\arrow[dd] \\
     & & S_0\arrow[rd] \arrow[rr] & & S\arrow[rd] & \\
     & & & X \arrow[rr] & & \mathcal{M}_{g,n} 
\end{tikzcd}\]
     
    For any map $u:S \to \mathcal{M}_{g,n}$, we can draw the above diagram, where each square is a pullback square. Since $T\subseteq C_X$ does not meet any section of $C_X\to X$, $S'\subseteq C_0$ also does not meet any section of $C_0\to S_0$. By \cref{prop:equiv} (3) and the proof of it, $p:S'\to S$ satisfies \cref{conj:mconj}. Since $X\to \mathcal{M}_{g,n}$ is finite flat and surjective, $S_0\to S$ too. Since $S'\to S_0$ is a base change of $T\to X$, this is also a  finite flat surjective map. Hence $p:S' \to S$ is a finite flat surjective map that satisfies \cref{conj:mconj}.
\end{proof}

In \cref{prop:general}, (1) suggests an important strategy of this paper: Instead of proving the lifting conjectures for every single family of curves, we aim to solve it `universally'. We will describe this approach in more detail in \cref{sec:LNS}.

\begin{rmk}\label{rmk:stack}
    \cref{prop:general} (3) is essentially equivalent to the following statement: The morphism $\mathcal{M}_{g,n+1}\to \mathcal{M}_{g,n}$ is an epimorphism of stacks on the category of $k$-schemes endowed with the finite flat topology.
\end{rmk}

\subsection{Applications of the lifting conjectures}\label{sec:App1Proof}
The idea of constructing a complete subvariety of $M_g$ can be described as follows: If one has a complete nondegenerate $d$-dimensional family $C\to S$ of smooth curves, this induces a complete non-degenerate $d+1$-dimensional family $C\times_S C\to C$ of curves with a marked point, which is given by the diagonal map. By \cref{conj:mconj} we can lift this to a complete nondegenerate $d+1$-dimensional family of curves with $n$ marked points, where the Kodaira construction can be applied.

Here, we apply this idea to prove the following results:
\begin{thm}\label{thm:mapp}
    Assume that \cref{conj:mconj} holds. Then
    \begin{enumerate}
        \item $r(g,n)=r(g)+1$ for all $n\ge 1$.
        \item Moreover, for any positive integer $d$, assume $g\ge 3\cdot 2^{d-1}$ if $\text{char }K\ne 2$, and $g\ge \frac{5\cdot 3^{d-1}+1}{2}$ if $\text{char }K=2$. Then $M_g$ contains a $d$-dimensional complete subvariety. 
    \end{enumerate}
\end{thm}

\begin{proof}
    (1) follows directly from \cref{conj:mconj} and the affineness (resp. properness) of the map $\mathcal{M}_{g,n+1}\to \mathcal{M}_{g,n}$ (resp. $\mathcal{M}_{g,1}\to \mathcal{M}_{g}$).

    (2) As the argument is essentially the same, we only prove it for $\text{char }K\ne 2$. Use induction on $d$. If $d=1$, as mentioned in \cref{Pthm:results}, this follows from the Satake compactification. Assume this holds for $d$. Let $S\to \mathcal{M}_{3\cdot 2^{d-1}}$ be a complete nondegenerate $d$-dimensional family of curves that correspond to such a subvariety of $M_{3\cdot 2^{d-1}}$. Then the corresponding family of curves $C\to S$ defines a complete nondegenerate $d+1$ dimensional family of the genus $3\cdot 2^{d-1}$ with a marked point. Applying \cref{conj:mconj} iteratively, we obtain a complete non-degenerate $d+1$ dimension family $C_n\to \mathcal{M}_{3\cdot 2^{d-1}, 2n}$ of genus $3\cdot 2^{d-1}$ with $2n$ marked points. Applying \cref{Kthm:kodaira} to $C_n$, we obtain a complete non-degenerate $d+1$ dimensional family of genus $3\cdot 2^{d}+n-1$. This proves the theorem. 
\end{proof}

\begin{rmk}\label{rmk:Zaal}
    Note that $g=\frac{5\cdot 3^{d-1}+1}{2}$ case of \cref{thm:mapp} is already known, as stated in \cite[Remark 1.13]{Za05}. However, the original contribution of \cref{thm:mapp} is that it holds in the characteristic $2$ case for $g\ge\frac{5\cdot 3^{d-1}+1}{2}$. Proving this requires the use of lifting theorems, such as \cref{conj:mconj}. 
\end{rmk}

\begin{thm}\label{thm:mapp2}
    Assume that \cref{conj:mconj2} holds. Then $\Tilde{r}(g,n)=\Tilde{r}(g)+1$ for all $n\ge 1$. In particular, for $g\ge 3$ and $n\ge 1$, any point in $\rm{M}_{g,n}$ is contained in a complete subsurface. 
\end{thm}

The first assertion also follows from the affineness (resp. properness) of the map $\mathcal{M}_{g,n+1}\to \mathcal{M}_{g,n}$ (resp. $\mathcal{M}_{g,1}\to \mathcal{M}_{g}$) and \cref{conj:mconj2}. The second assertion is then a corollary of \cref{Pthm:results} (1). Note that, by the Satake compactification, any point of $\rm{M}_g$ is contained in a complete subcurve. This and \cref{conj:mconj2} implies the second assertion. In \S \ref{sec:charp}, we will prove \cref{thm:mapp} and \cref{thm:mapp2} holds for positive characteristic.

\section{Line bundles detecting a new section}\label{sec:LNS}

We will reduce \cref{conj:mconj} to a problem of finding line bundles with certain properties.

\begin{defn}\label{defn:dnsline}
Let $f:C\to S$ be a family of smooth curves with $n$ marked points $s_1,\cdots, s_n:S\to C$ over a finite type $K$-scheme $S$. A line bundle $\mathcal{L}$ on $C$ is a \textbf{DNS line bundle} \footnote{abbreviation of `Detecting a New Section'} if there exists a nontrivial effective Cartier divisor $T$ of $C$ such that 
\begin{enumerate}
    \item  $\mathcal{L}^{\otimes m}\simeq \mathcal{O}(T)$, where $\mathcal{O}(T)$ is the dual of the ideal sheaf of $T$.
    \item  The support of $T$ does not intersect with any of $S_1,\cdots, S_n$. 
\end{enumerate}
\end{defn}

Note that $\mathcal{L}$ is a DNS line bundle if and only if $\mathcal{L}^{\otimes m}$ is for some $m>0$. This definition is simply a translation of the notion of a `new section' into the language of line bundles.

By \cref{prop:equiv} (3), proving \cref{conj:mconj} for $f:C\to S$ is essentially equivalent to finding a DNS line bundle $\mathcal{L}$ on $C$. 

\begin{prop}\label{prop:dnsconj}
    If $S$ is a variety, then the existence of a DNS line bundle implies \cref{conj:mconj}. Moreover, if $S$ is smooth, then the converse holds.
\end{prop}

The first assertion holds by \cref{prop:equiv} (3). If \cref{prop:equiv} (3) holds for smooth variety $S$, then $T$ is an effective Cartier divisor by the smoothness of $C$, hence induces a DNS line bundle. The smoothness of $S$ is not a serious assumption, by \cref{prop:general} (2).

Note that in \cite{Za99}, Zaal essentially proved \cref{Pthm:Zathm} by showing that on the family of smooth curves $C\times_S C\to C$ with a marked point given by the diagonal map $\Delta:C\to C\times_S C$, associated to a family of smooth curve $C\to S$ over a proper smooth curve $S$, $\pi_1^\ast\Omega_{C/S}^1\otimes\pi_2^\ast\Omega_{C/S}^1 \otimes \mathcal{O}(2\Delta)$ is a DNS line bundle on $C\times_S C$.

Before beginning to search for DNS line bundles, we will first list some facts about them. 

\begin{prop}\label{prop:dnsprop}
    Let $f:C\to S$ and $s_i:S\to C$ as above. Assume $S$ is proper. If $s_i^\ast\mathcal{L}$ is trivial for all $1\le i\le n$ and $\mathcal{L}$ is semiample, then $\mathcal{L}$ is a DNS line bundle.
\end{prop}

\begin{proof}
Choose $m$ such that $\mathcal{L}^{\otimes m}$ is base point free and let $h: C\to \mathbb{P}^k$ be the corresponding morphism. Since $s_i^\ast\mathcal{L}$ is trivial and $S$ is proper, $h|_{S_i}$ is constant on each connected component. Therefore, the image $h\left(\cup_i S_i \right)$ is a set of finitely many points. Hence by taking a hyperplane section that does not contain $h\left(\cup_i S_i \right)$, we can construct a nontrivial effective Cartier divisor equivalent to $\mathcal{L}^{\otimes m}$, satisfying the condition for the DNS line bundle.
\end{proof}

\begin{rmk}\label{rmk:dnsuse}
    One can prove the first assertion of \cref{thm:themain} using \cref{prop:dnsprop}, and consequently establish \cref{thm:mainthm1} and \cref{thm:mainthm2}. In particular, it is not necessary to employ a further theory of DNS line bundles. By applying \cref{thm:omegasemi} and performing calculations similar to those in \cref{prop:linecal}, we can demonstrate that $\Omega_{C/S}^1(S_1+\cdots+S_n)$ satisfies the assumptions of \cref{prop:dnsprop} when $\text{char }K\ne 0$. Therefore, if $\text{char }K\ne 0$, we can establish \cref{conj:mconj} for proper varieties $S$. To demonstrate this for general varieties, we need to use a compactification process, as in the proof of \cref{thm:seminds}. We can utilize the proof of \cref{thm:conj2prf} with $\Omega_{C/S}^1(S_1+\cdots+S_n)$ to establish \cref{conj:mconj2}, again using \cref{thm:omegasemi}. However, in this case, the `universal' approach makes the proof much simpler since the exceptional locus of a line bundle over a family of curves may be complicated due to degeneration, in contrast to the result in \cref{thm:omegasemi}. Note that, to prove this for general varieties, we also require a compactification process.

    Nevertheless, it is still necessary to develop the theory of DNS line bundles for several reasons. Firstly, this is essential for proving the second assertion of \cref{thm:themain}. Secondly, it enables us to take a systematic approach to lifting conjectures, even in characteristic $p$. For instance, \cref{thm:seminds} and \cref{thm:conj2prf} highlight the fundamental components of the proof of the lifting conjecture. Thirdly, this notion, especially the concept of the universal DNS line bundle (\cref{defn:univdns}), establishes a connection between the problem of complete subvarieties and the line bundles on the moduli space of curves. For example, see 
    \cref{prob:level}, \cref{conj:onlyhope}, \cref{thm:psisemiconcl}, and the discussions that follow them.
\end{rmk}

\begin{prop}\label{prop:dnspullback}
    Let $f:C\to S$ be a family of smooth pointed curves over a finite type $K$-scheme $S$ and $p:S'\to S$ be a morphism between finite type $K$-schemes, and $f':C'\to S'$ be the pullback of $f$ by $p$, and $p':C'\to C$ be the induced map between the family of curves.
    \begin{enumerate}
        \item If $\mathcal{L}$ is a DNS line bundle, then $(p')^\ast\mathcal{L}$ is also a DNS line bundle.
        \item Moreover, assume that $p$ is finite,  surjective, flat and $S$ is normal. Then the converse of (1) holds.
    \end{enumerate} 
\end{prop}

\begin{proof}
    (1) Choose $m>0$ such that $\mathcal{L}^{\otimes m}\simeq \mathcal{O}(T)$, $T$ does not intersect with $S_i$'s. Then 
    \[ (p')^\ast\mathcal{L}^{\otimes m}\simeq (p')^\ast\mathcal{O}(T)=\mathcal{O}\left((p')^\ast (T) \right)  \]
    Since we can pull back the effective Cartier divisor in this case. The support of $(p')^\ast (T)$ does not intersect with the pullback of $S_i$'s, hence $(p')^\ast\mathcal{L}$ is a DNS line bundle.

    (2) Let $s_i'$ be the pullback of $s_i$ and $d$ be the degree of $p$. By assumption, there exists $m>0$ such that $(p')^\ast\mathcal{L}^{\otimes m}\simeq \mathcal{O}(T')$ for some effective Cartier divisor $T'$ and the support of $T'$ does not intersect with $S_i'$, $1\le i\le n$. Let $c_1\left(\mathcal{L}\right)=\sum_i n_i T_i \in \text{CH}^1(C)$ where $T_i$ are prime Weil divisors. Then $(p')^\ast\mathcal{L}^{\otimes m}\simeq \mathcal{O}(T')$, so by taking their associated Weil divisor, $[T']$ and $\sum_i mn_i\cdot (p')^\ast [T_i]$ represent the same element in $\text{CH}^1(C')$. Take $p'_\ast$ from each divisor class. This then implies that $\sum_i dmn_i [T_i]$ and $p'_\ast [T]$ are rationally equivalent as Weil divisors. Note that $C$ is a normal scheme. Since $\sum_i dmn_i [T_i]$ is locally principal, $p'_\ast [T]$ is also a locally principal effective Weil divisor, hence $p'_\ast [T]$ is an effective Cartier divisor. The natural map $\text{Pic}(C)\to \text{CH}^1(C)$ is injective, hence by taking the associated line bundles, we find that $\mathcal{L}^{\otimes md}$ is equivalent to $\mathcal{O}\left(p'_\ast [T]\right)$. Note that $p'_\ast [T]$ is an effective Cartier divisor supported on $p'(T)$. Since the support of $T$ does not intersect with the $S_i'$'s, the support of $p'(T)$ does not intersect with the $S_i$'s. Hence $\mathcal{L}$ is a DNS line bundle.
\end{proof}

As mentioned in \cref{prop:dnsconj}, proving \cref{conj:mconj} for $f: C\to S$ is essentially equivalent to finding DNS line bundles on $S$. Hence, the first step in solving the conjecture should be to search for possible candidates for DNS line bundles. For a specific family of curves $C\to S$, there may be many line bundles on $C$, and some of them may be DNS. However, we need to find a DNS line bundle for every such family of curves. Hence, the natural choice is a line bundle that comes from the universal curve $\mathcal{U}_{g,n}$ over $\mathcal{M}_{g,n}$. This idea suggests the following definition.

\begin{defn}\label{defn:univdns}
    A line bundle $\mathcal{L}\in \text{Pic}(\mathcal{U}_{g,n})$ is a \textbf{universal DNS line bundle} if for any pullback diagram
    \[\begin{tikzcd}
    C \arrow[r, "v"]\arrow[d, "f"]& \mathcal{U}_{g,n}\arrow[d, "\pi"]\\
    S  \arrow[r, "u"]& \mathcal{M}_{g,n}
\end{tikzcd}\]
    corresponds to a family of smooth curves $C\to S$ over a normal variety $S$, $v^\ast\mathcal{L}$ is a DNS line bundle. $\mathcal{L}$ is a \textbf{weakly universal DNS line bundle} if the same holds for every proper normal variety $S$.
\end{defn}

We have included the normality condition in order to make use of \cref{prop:dnspullback} (2). This is just a technical requirement. We could refine the definition of the DNS line bundle to make \cref{prop:dnspullback} (2) work even when $S$ is not normal. However, this would add unnecessary complexity to both the definition and subsequent arguments (including both Weil and Cartier divisors). Therefore, we have decided to stick with the current definition.

By \cref{prop:dnsconj} and the finiteness of the normalization map for varieties, the existence of a universal DNS line bundle (resp. weakly universal line bundle) implies \cref{conj:mconj} (resp. \cref{conj:mconj} for proper varieties). The (weakly) universal DNS line bundle can be regarded as a geometric mechanism that finds a new section of any family of curves over a (proper) variety. Therefore, our goal is to find a universal DNS line bundle. Fortunately, we know the Picard group of $\mathcal{U}_{g,n}$.

We recall several line bundles over $\overline{\mathcal{M}}_{g,n}$. To define a line bundle over $\overline{\mathcal{M}}_{g,n}$, it is enough to functorially define a line bundle on $S$ for each family $f: C\to S$ of stable curves of genus $g$ with $n$ marked points $s_1,\cdots, s_n$. The \textbf{Hodge line bundle} $\lambda$ is defined by $\text{det} f_\ast \Omega_{C/S}^1$. The \textbf{psi claseses} $\psi_1,\cdots, \psi_n$ are defined by $\psi_i=s_i^\ast \Omega_{C/S}^1$. Moreover, there are line bundles $\delta_S$ on $\overline{\mathcal{M}}_{g,n}$ defined by boundary divisors, called \textbf{boundary classes}. Here, $S$ is either `irr' or a pair $(a,P)$ of a number $0\le a\le g$ and a subset $P$ of $\left\{1,\cdots, n \right\}$. They correspond to the topological types of singular stable curves of genus $g$. See \cite{AC87, AC09} for a detailed explanation of these classes and the relations between them. Also, note that the universal curve $\overline{\mathcal{U}}_{g,n}$ over $\overline{\mathcal{M}}_{g,n}$ is isomorphic to $\overline{\mathcal{M}}_{g,n+1}$, and $\mathcal{U}_{g,n}$ is an open substack of $\overline{\mathcal{M}}_{g,n+1}$ is the complement of boundary except the components correspond to $\delta_{0,\left\{1,n+1\right\}},\cdots, \delta_{0,\left\{n,n+1\right\}}$. This will be used multiple times to examine the properties of line bundles.

\begin{thm}\cite[Theorem 2]{AC87}\label{thm:acline}
    Assume $g\ge 3$. Then $\text{Pic}(\mathcal{M}_{g,n})$ is  freely generated by $\lambda, \psi_1,\cdots, \psi_n$. $\text{Pic}(\overline{\mathcal{M}}_{g,n})$ is freely generated by $\lambda, \psi_1,\cdots, \psi_n$ and $\delta_S$'s.
\end{thm}

The following corollary is a direct consequence of \cref{thm:acline}.

\begin{cor}\label{cor:corline}
    Assume $g\ge 3$. Then $\lambda, \psi_1,\cdots, \psi_{n+1}$ and $\delta_{0,\left\{1,n+1\right\}},\cdots, \delta_{0,\left\{n,n+1\right\}}$ freely generate $\text{Pic}(\mathcal{U}_{g,n})$.
\end{cor}

Now, we will determine which line bundles of $\text{Pic}(\mathcal{U}_{g,n})$ are the candidates for DNS line bundles. The following proposition, proved mainly in \cite{AC09} (see \cref{rmk:psi}), summarizes every calculation we need. 

\begin{prop}\label{prop:linecal}
    \begin{enumerate}
        \item Let $f:C\to S$ be a family of stable curves with $n$ marked points $s_1,\cdots, s_n:S\to C$ over any $K$-scheme $S$. Let 

    \[ \begin{tikzcd}
        C\arrow[r, "v"]\arrow[d, "f"] &  \overline{\mathcal{U}}_{g,n} \arrow[d, "\pi"]\\
            S\arrow[r, "u"] &  \overline{\mathcal{M}}_{g,n}.
    \end{tikzcd} \]

    be the diagram induced from the universal property of $\overline{\mathcal{U}}_{g,n}$. Then
    \begin{align*}
        &v^\ast \lambda=\text{det} f^\ast f_\ast \Omega_{C/S}^1\\
        &v^\ast \psi_i=\mathcal{O}(S_i)\otimes f^\ast s_i^\ast \Omega_{C/S}^1 \text{ for }1\le i\le n\\
        &v^\ast \psi_{n+1}=\Omega^1_{C/S}(S_1+\cdots+S_n)\\
        &v^\ast \delta_{0,\left\{i,n+1\right\}}=\mathcal{O}(S_i).
    \end{align*}

    Note that the third row is equivalent to $\psi_{n+1}\simeq \Omega_{\pi}\left(\overline{S}_1+\cdots+\overline{S}_n \right)$.

    \item Let $\bar{s}_1,\cdots, \bar{s}_n:\overline{\mathcal{M}}_{g,n}\to \overline{\mathcal{U}}_{g,n}$ be universal sections. Then 
    \begin{multicols}{2}
      \begin{enumerate}[itemsep=0.3cm]
        \item $\bar{s}_i^\ast \lambda=\lambda$
        \item $\bar{s}_i^\ast \psi_j=\psi_j  \text{ for }j\ne i,n+1$
        \item $\bar{s}_i^\ast \psi_i=\mathcal{O}_{\mathcal{M}_{g,n}} $
        \item $\bar{s}_i^\ast \psi_{n+1}=\mathcal{O}_{\mathcal{M}_{g,n}}$
        \item $\bar{s}_i^\ast \delta_{0,\left\{j,n+1\right\}}=\mathcal{O}_{\mathcal{M}_{g,n}} \text{ for }j\ne i$
        \item $\bar{s}_i^\ast \delta_{0,\left\{i,n+1\right\}}=-\psi_i$
      \end{enumerate} 
    \end{multicols}
    
    \end{enumerate}
    
\end{prop}
\begin{proof}
    (1) By \cite[Lemma 1 (i)-(iv)]{AC09} and Mumford's identity $\kappa_1=12\lambda-\delta+\psi$, we have $\pi^\ast\lambda=\lambda$. Therefore, 
    \[ v^\ast\lambda=v^\ast\pi^\ast\lambda=f^\ast u^\ast \lambda =f^\ast\det f_\ast \Omega_{C/S}^1=\det f^\ast f_\ast \Omega_{C/S}^1. \]
    Since $\delta_{0,\left\{i,n+1 \right\}}$ is the image of $\bar{s}_i$, $v^\ast \delta_{0,\left\{i,n+1\right\}}=\mathcal{O}(S_i)$. Also, by \cite[Lemma 1 (ii)]{AC09}, $\psi_i=\pi^\ast\psi_i+\delta_{0,\left\{i,n+1 \right\}}$. Hence
    \[ v^\ast \psi_i=v^\ast\left(\pi^\ast\psi_i+\delta_{0,\left\{i,n+1 \right\}}\right)=f^\ast u^\ast \psi_i \otimes \mathcal{O}(S_i)=\mathcal{O}(S_i)\otimes f^\ast s_i^\ast \Omega_{C/S}^1. \]
    Now we need to prove $\psi_{n+1}\simeq \Omega_{\pi}\left(\overline{S}_1+\cdots+\overline{S}_n \right)$. We will identify $\overline{\mathcal{U}}_{g,n+1}$ with $\overline{\mathcal{M}}_{g,n+1}$.

    Let $\mathcal{Z}_{g,n+1}$ be the closed substack of $\overline{\mathcal{M}}_{g,n+1}$ corresponds to the locus of pointed curves $(C,P_1,\cdots, P_{n+1})$ such that $P_{n+1}$ is contained in a rational component of $C$, which consists of $P_{n+1}$, two nodes, and no other special points. Let $\mathcal{V}_{g,n+1}$ be the open substack which is complement of $\mathcal{Z}_{g,n+1}$. By an elementary dimension calculation, we have that $\mathcal{Z}_{g,n+1}$ is codimension $2$. Let $\overline{\mathcal{M}}_{g,n+1}\times_{\overline{\mathcal{M}}_{g,n}} \overline{\mathcal{M}}_{g,n+1}$ be the fiber product of $\pi: \overline{\mathcal{M}}_{g,n+1}\to \overline{\mathcal{M}}_{g,n}$ with itself. Then $\pi_1: \overline{\mathcal{M}}_{g,n+1}\times_{\overline{\mathcal{M}}_{g,n}} \overline{\mathcal{M}}_{g,n+1}\to \overline{\mathcal{M}}_{g,n+1}$ has sections $\bar{s}_1,\cdots, \bar{s}_n$, which are the base changes of $\bar{s}_1,\cdots, \bar{s}_n:\overline{\mathcal{M}}_{g,n}\to \overline{\mathcal{M}}_{g,n+1}$, and the diagonal map $\Delta$. Let $\mathcal{M}'$ be the blow-up of $\overline{\mathcal{M}}_{g,n+1}\times_{\overline{\mathcal{M}}_{g,n}} \overline{\mathcal{M}}_{g,n+1}$ by the intersections of $\Delta$ and $\bar{s}_i$, for $1\le i\le n$. We have a following description of $\beta:\mathcal{M}'\to \overline{\mathcal{M}}_{g,n+1}$: on $\mathcal{V}_{g,n+1}$, this is the same as the projection map $\overline{\mathcal{M}}_{g,n+2}\to \overline{\mathcal{M}}_{g,n+1}$ and on $\mathcal{Z}_{g,n+1}$, the fiber of this map on $(C,P_1,\cdots, P_{n+1})$ is the stabilization of $(C,P_1,\cdots, P_{n})$. By this description, there exists a natural (partial stabilization) map $\alpha:\overline{\mathcal{M}}_{g,n+2}\to \mathcal{M}'$. All in all, we have the following diagram.
    
    \[\begin{tikzcd}
       \overline{\mathcal{M}}_{g,n+2}\arrow[r, "\alpha"] & \mathcal{M}'\arrow[r, "\beta"] & \overline{\mathcal{M}}_{g,n+1}\times_{\overline{\mathcal{M}}_{g,n}} \overline{\mathcal{M}}_{g,n+1}\arrow[r, "\pi_2"]\arrow[d, "\pi_1"] & \overline{\mathcal{M}}_{g,n+1}\arrow[d, "\pi"] \\
         &  & \overline{\mathcal{M}}_{g,n+1}\arrow[r]\arrow[ul, swap, "\Delta_1"]\arrow[ull, "\Delta_2"] & \overline{\mathcal{M}}_{g,n}
    \end{tikzcd}\]
    where $\Delta_1$ is the proper transformation of $\Delta$, and $\Delta_2$ is the $(n+1)$st section. Let $\pi_3=\pi_1\circ\beta\circ\alpha:\overline{\mathcal{M}}_{g,n+2}\to \overline{\mathcal{M}}_{g,n+1}$ and $\pi_4=\pi_1\circ\beta:\mathcal{M}'\to \overline{\mathcal{M}}_{g,n+1}$ be the natural projections. We will show that $\Delta_1^\ast \Omega_{\pi_4}=\Delta_2^\ast \Omega_{\pi_3}$. Since they are line bundles, it is enough to check that they agree on $\mathcal{V}_{g,n+1}$. Consider the series of maps $\overline{\mathcal{M}}_{g,n+2}\to \mathcal{M'}\to \overline{\mathcal{M}}_{g,n+1}$. As previously mentioned, if we restrict the first map to $\mathcal{V}_{g,n+1}$, then $\alpha$ is an isomorphism. Hence, if we restrict everything to $\mathcal{V}_{g,n+1}$, we have $\Delta_1^\ast \Omega_{\pi_4}=\Delta_2^\ast \Omega_{\pi_3}$. Hence $\Delta_1^\ast \Omega_{\pi_4}=\Delta_2^\ast \Omega_{\pi_3}$ on $\overline{\mathcal{M}}_{g,n+1}$. Since $\psi_{n+1}=\Delta_2^\ast \Omega_{\pi_3}$ by definition, we only need to calculate $\Delta_1^\ast \Omega_{\pi_4}$. Note that, by the formula relating the relative canonical sheaves of blow-ups, we have $\Omega_{\pi_4}=\beta^\ast\Omega_{\pi_1}(E_1+\cdots+E_n)$ where $E_i$ are exceptional divisors correspond to $\Delta\cap \bar{S}_i$. Therefore, $\Delta_1^\ast \Omega_{\pi_4}=\Omega_{\pi}(\overline{S}_1+\cdots+\overline{S}_n)$.

    (2) Note that we have a pullback diagram
    \[ \begin{tikzcd}
        S\arrow[r, "u"]\arrow[d, "s_i"] &  \overline{\mathcal{M}}_{g,n} \arrow[d, "\bar{s}_i"]\\
            C\arrow[r, "v"] &  \overline{\mathcal{U}}_{g,n}.
    \end{tikzcd} \]
    It is enough to prove that this holds after pulling everything back by $u$. In fact, by (1),
    \begin{align*}
        &u^\ast\bar{s}_i^\ast \lambda=s_i^\ast v^\ast\lambda=s_i^\ast\text{det} f^\ast f_\ast \Omega_{C/S}^1=\text{det} f_\ast \Omega_{C/S}^1 =u^\ast\lambda\\
        &u^\ast\bar{s}_i^\ast \psi_j=s_i^\ast v^\ast\psi_j=s_i^\ast \left(\mathcal{O}(S_j)\otimes f^\ast s_j^\ast \Omega_{C/S}^1\right)=s_j^\ast \Omega_{C/S}^1=u^\ast\psi_j \text{ if }i\ne j\\
        &u^\ast\bar{s}_i^\ast \psi_i=s_i^\ast v^\ast\psi_i=s_i^\ast \left(\mathcal{O}(S_i)\otimes f^\ast s_i^\ast \Omega_{C/S}^1\right)=\mathcal{O}_S=u^\ast\mathcal{O}_{\mathcal{M}_{g,n}} \\
        &u^\ast\bar{s}_i^\ast \psi_{n+1}=s_i^\ast v^\ast  \psi_{n+1}=s_i^\ast \Omega^1_{C/S}(S_1+\cdots+S_n)=\mathcal{O}_S=u^\ast\mathcal{O}_{\mathcal{M}_{g,n}}\\
        &u^\ast\bar{s}_i^\ast \delta_{0,\left\{i,n+1\right\}}=s_i^\ast v^\ast  \delta_{0,\left\{i,n+1\right\}}=s_i^\ast \mathcal{O}(S_i)=-s_i^\ast \Omega^1_{C/S}=-u^\ast\psi_i\\
        &u^\ast\bar{s}_i^\ast \delta_{0,\left\{j,n+1\right\}}=s_i^\ast v^\ast \delta_{0,\left\{j,n+1\right\}}=s_i^\ast \mathcal{O}(S_i)=\mathcal{O}_S=u^\ast\mathcal{O}_{\mathcal{M}_{g,n}} \text{ for }j\ne i.
    \end{align*}
    In the third, fourth and fifth line, we used $s_i^\ast \mathcal{O}(S_i)=-s_i^\ast \Omega^1_{C/S}$ which follows from the standard intersection theory. Therefore, (2) holds.
\end{proof}

\begin{rmk}\label{rmk:psi}
    The identity $v^\ast \psi_{n+1}=\Omega^1_{C/S}(S_1+\cdots+S_n)$ is frequently referred to in the literature, see e.g. \cite[Lemma 1.4, Sections 2.B and 2.C]{Ka93} and \cite[Page 162]{AC87}. The proof of this identity can be found in \cite[Theorem 3.21]{Zv12}. However, despite its importance and versatility, there are only a few published articles that contain proof of this, so we have decided to provide proof of it here.
\end{rmk}

The following is the starting point of the classification of universal DNS line bundles.

\begin{lem}\label{lem:univline}
    Assume $g\ge 3$. Let $\mathcal{L}\in \text{Pic}\left(\mathcal{M}_{g,n} \right)$ such that for any normal variety $S$ and morphism $f:S\to \mathcal{M}_{g,n}$, there exists $m>0$ such that $f^\ast \mathcal{L}^{\otimes m}$ is trivial. Then $\mathcal{L}$ is trivial.
\end{lem}

\begin{proof}
    Let $c_X: X\to \mathcal{M}_{g,n}$ be the finite etale surjective map from a smooth variety $X$, stated in \cref{lem:covmod} (1). Let $d$ be the degree of $c_X$ and $m$ be a positive integer such that $c_X^\ast \mathcal{L}^{\otimes m}$ is trivial on $X$. For any $K$-scheme $S$ and a map $u:S\to \mathcal{M}_{g,n}$. Let $S'=S\times_{\mathcal{M}_{g,n}}X$ and $p:S'\to S$, $q:S'\to X$ be the projection maps. We have the following.
    
    \[ p_\ast \mathcal{O}_{S'} \otimes u^\ast \mathcal{L}^{\otimes m}=p_\ast p^\ast u^\ast \mathcal{L}^{\otimes m}=p_\ast q^\ast c_X^\ast \mathcal{L}^{\otimes m}=p_\ast q^\ast \mathcal{O}_X=p_\ast \mathcal{O}_{S'}.   \]
    
    Since $p$ is a finite etale morphism of degree $d$, $p_\ast \mathcal{O}_{S'}$ is a rank $d$ vector bundle. Hence by taking the determinant, we know $u^\ast \mathcal{L}^{\otimes md}$ is a trivial line bundle. This holds for every $S$ and morphism $u$, $\mathcal{L}^{\otimes md}$ is trivial. since $\text{Pic}(\mathcal{M}_{g,n})$ is torsion-free for $g\ge 3$, $\mathcal{L}$ is trivial. 
\end{proof}

\begin{thm}\label{thm:ndsclas}
Assume $g\ge 3$. Let $\mathcal{L}\in \text{Pic}(\mathcal{U}_{g,n})$ be a universal DNS line bundle.
    \begin{enumerate}
    \item If $n\ge 2$, then $\mathcal{L}$ is a multiple of $\psi_{n+1}$.

    \item If $n=1$, then $\mathcal{L}$ is a linear combination of $\psi_1$ and $\psi_2$.
    \end{enumerate}   
\end{thm}

\begin{proof}
    We will prove that for any normal variety $S$ and a morphism $u:S \to \mathcal{M}_{g,n}$, there exists $m>0$ such that $u^\ast \bar{s}_i^\ast\mathcal{L}^{\otimes m}$ is trivial: Let $f: C\to S$ be the family of smooth curves corresponding to $u$ and $v: C\to \mathcal{U}_{g,n}$ be the morphism of families of curves. Then, by the definition of DNS line bundle, there exists $m>0$ such that $u^\ast \bar{s}_i^\ast\mathcal{L}^{\otimes m}=s_i^\ast v^\ast\mathcal{L}^{\otimes m}$ is trivial. Hence, by \cref{lem:univline}, $\bar{s}_i^\ast\mathcal{L}$ is trivial. We will prove (1) and (2) using this fact.

    By \cref{cor:corline}, $\mathcal{L}=a\lambda+\sum_{j=1}^{n+1}b_j\psi_j+\sum_{k=1}^n c_k \delta_{0,\left\{k, n+1\right\}}$. By \cref{prop:linecal} (2), 
    \[ \bar{s}_i^\ast\mathcal{L}=a\lambda+\sum_{1\le j\le n, j\ne i}b_j\psi_j-c_i\psi_i=0. \]
    If $n\ge 2$, this is equivalent to $a=b_j=c_j=0$ for $1\le j\le n$, so $\mathcal{L}$ is a multiple of $\psi_{n+1}$. If $n=1$, this is equivalent to $a=c_1=0$. Hence $\mathcal{L}$ is a linear combination of $\psi_1$ and $\psi_2$.
\end{proof}

Therefore, we have some candidates for universal DNS line bundles. In the next section, we will examine which of them are indeed universal DNS line bundles. The answer depends on the characteristic of the base field. This reveals the difference between the geometry of the moduli space of curves in characteristic $p$ and characteristic $0$.

As can be seen from \cref{prop:dnsprop}, the universal DNSness is related to the semi-ampleness of the line bundle.

\begin{thm}\label{thm:seminds}
    Let $\mathcal{L}\in \text{Pic}(\mathcal{U}_{g,n})$, $\mathcal{L}'\in \text{Pic}\left(\overline{\mathcal{U}}_{g,n}\right)$ be line bundles and $i: \mathcal{U}_{g,n}\to \overline{\mathcal{U}}_{g,n}$ be the canonical inclusion. 

    \begin{enumerate}
        \item If $\mathcal{L}$ is semiample and $\bar{s}_i^\ast\mathcal{L}$ are trivial, then $\mathcal{L}$ is a weakly universal DNS line bundle.
        \item If $\mathcal{L}'$ is semiample and $\bar{s}_i^\ast\mathcal{L}'$ are trivial, then $i^\ast\mathcal{L}'$ is a universal DNS line bundle.
    \end{enumerate}
\end{thm}
\begin{proof}
    (1) is just a corollary of \cref{prop:dnsprop}. The basic strategy for (2) is the same as \cref{prop:dnsprop}. However, since the base scheme $S$ of a family of curves $f: C \to S$ can be nonproper, the global section of the trivial line bundle may not be $K$. Hence, we will first compactify $S$ and use the same method. By \cref{lem:covmod} (2), there exists a $K$-scheme $\overline{Y}$ with a finite flat surjective map $\overline{Y}\to \overline{\mathcal{M}}_{g,n}$. This implies $\overline{Y}$ is a projective $K$-scheme. Let $Y:=\overline{Y}\times_{\overline{\mathcal{M}}_{g,n}}\mathcal{M}_{g,n}$ be an open subscheme of $\overline{Y}$. Then $Y\to \mathcal{M}_{g,n}$ is also a finite flat surjective map. Fix a family $f: C\to S$ of smooth curves of genus $g$ with $n$ marked points over a normal variety $S$. Let $S\to \mathcal{M}_{g,n}$ be the corresponding map and define $S':=S\times_{\mathcal{M}_{g,n}} Y$. Then $S'\to S$ is also finite, flat, and surjective. By taking the Nagata compactification with respect to the map $S'\to \overline{Y}$, we obtain a proper $K$-scheme $\overline{S'}$ with an open immersion $j:S' \to \overline{S'}$ such that $S'\to \overline{Y}$ factors through $\overline{S'}$ (note that $\overline{S'}$ may be neither irreducible nor reduced). Hence we have the following diagram.
    \[ \begin{tikzcd}
    C\arrow[dd, "f"]\arrow[rd, "v"] & & C'\arrow[dd, near start, "f'"]\arrow[ll]\arrow[rr] & &\overline{C'}\arrow[dd, near start, "\bar{f}'"]\arrow[rd, "\bar{v}'"] & \\
    &\mathcal{U}_{g,n}\arrow[dd,  near start, "\pi"]\arrow[rrrr, "i"] & & & &\overline{\mathcal{U}}_{g,n}\arrow[dd, "\pi"] \\
    S\arrow[rd] & &S'\arrow[rr, "j"]\arrow[ll] & &\overline{S'}\arrow[rd] & \\
    &\mathcal{M}_{g,n}\arrow[rrrr, "i"] & & & &\overline{\mathcal{M}}_{g,n}
    \end{tikzcd} \]
    We have already defined the bottom rectangle. The upper rectangle corresponds to the family of curves $f:C\to S$, $f':C'\to S'$, and $\bar{f}':\overline{C'} \to \overline{S'}$. Let $v:C\to \mathcal{U}_{g,n}$, $v':C'\to \mathcal{U}_{g,n}$,and $\bar{v}':\overline{C'}\to \overline{\mathcal{U}}_{g,n}$ be the morphism associated to the universal curve. Under the assumption, $\left(\bar{v}'\right)^\ast \mathcal{L}'^{\otimes m}$ is base-point free, for some $m>0$. Let $h:\overline{C'}\to \mathbb{P}^k$ be the corresponding morphism. Since $\bar{s}_i^\ast\mathcal{L}'$ are trivial, the restriction of $\left(\bar{v}'\right)^\ast \mathcal{L}'^{\otimes m}$ to the sections of $\bar{f}'$ is trivial. Hence $h$ contracts them to a finite set of closed points. Consider the restriction $h|_{C'}:C'\to \mathbb{P}^k$, which is the map corresponding to the line bundle $v'^\ast \mathcal{L}'^{\otimes m}$. Then $h|_{C'}$ also contracts sections to a finite set of closed points $P\subseteq \mathbb{P}^k$. By the same method as \cref{prop:dnsprop}, $v'^\ast \mathcal{L}'$ is a DNS line bundle of $f':C'\to S'$. Note that $S'\to S$ is a finite flat surjective map, and $v'^\ast \mathcal{L}'$ is the restriction of $\mathcal{L}$ to $C'$. By \cref{prop:dnspullback} (2),  since $S$ is a normal variety, $v^\ast\mathcal{L}$ is a DNS line bundle on $f:C\to S$. Hence $\mathcal{L}$ is a universal DNS line bundle.
\end{proof}

To prove \cref{conj:mconj2}, we need more than universal DNS-ness. 

\begin{thm}\label{thm:conj2prf}
    With the notation of \cref{thm:seminds}, in particular, if $\mathcal{L}$ satisfies the condition of \cref{thm:seminds} (2) (resp. (1)), and its exceptional locus is contained in the boundary $\partial \mathcal{M}_{g,n}$, then \cref{conj:mconj2} (resp. \cref{conj:mconj2} for proper varieties $S$) holds if we identify $\overline{\mathcal{U}}_{g,n}$ with $\overline{\mathcal{M}}_{g,n+1}$. 
\end{thm}

For the definition of exceptional locus of a nef line bundle, see \cite[Definition 0.1]{Ke99}.

\begin{proof}
    Since the proofs for both cases are almost the same, we will prove the claim assuming the condition of \cref{thm:seminds} (2). Let $f: C\to S$ be a family of smooth curves of genus $g$ with $n$ marked points $s_1,\cdots,s_n:S \to C$ over an irreducible, finite type $K$-scheme $S$. By the proof of \cref{thm:seminds}, we may assume that $f:C \to S$ is an open subfamily of a family $\bar{f}:\overline{C}\to \overline{S}$ of stable genus $g$ curves with $n$ marked points $\bar{s}_1,\cdots,\bar{s}_n:\overline{S} \to \overline{C}$ over a proper connected $K$-scheme $\overline{S}$ after a finite flat base change. Let $v$ (resp. $\bar{v}$) be the corresponding map from $C$ (resp. $\overline{C}$) to $\mathcal{U}_{g,n}$ (resp. $\overline{\mathcal{U}}_{g,n}$). Then $\bar{v}^\ast\mathcal{L}$ is a semiample line bundle. Let $h:\overline{C}\to Z$ be the corresponding morphism into a projective $K$-scheme $Z$. By the Stein factorization, we may assume that every fiber of $h$ is connected. Since $\bar{s}_i^\ast \mathcal{L}$ are trivial and $\overline{S}$ is a proper connected $K$-scheme, $h\circ \bar{s}_i$ is a constant map. Let $x_i$ be the image of $h\circ \bar{s}_i$.

    We will show that if a closed point $x\in C$ is not contained in the image of $s_1,\cdots, s_n$, then $h(x)\ne x_i$ for every $1\le i\le n$. Assume $h(x)=x_i$ for some $x_i$. Since all fibers of $h$ are connected, $h^{-1}(x_i)$ is connected. This set contains $x$ and $S_i$. By connectedness, there exists a closed subcurve $D$ of $\overline{C}$ that satisfies $h(D)=x_i$, $x\in D$, and the intersection of $D$ and the image of $\bar{s}_i$ is nonempty. Since $h(D)=x_i$, $\mathcal{L}|_D^{\otimes m}$ is trivial for some $m>0$. Therefore, $\bar{v}^\ast\mathcal{L}\cdot D=0$, so $\bar{v}(D)$ is contained in the exceptional locus of $\mathcal{L}$ or it is a point. The first implies that $\bar{v}(D)$ is contained in the boundary. However, $\bar{v}(x)$ is not contained in the boundary, so cannot happen. Moreover, $\bar{v}(D)$ cannot be a point since $\bar{v}(x)$ is not contained in the boundary but $\bar{v}(D\cap S_i)$ is contained in the boundary. This is a contradiction.

    Now we will prove \cref{prop:equiv2} (3). For any $x$ not contained in the image of $s_1,\cdots, s_n$, $h(x)\ne x_i$ for all $i$ according to the above paragraph. Since $Z$ is a projective $K$-scheme, we can take the hyperplane section which contains $h(x)$ but does not contain $x_i$. The inverse image of this hyperplane section is the desired $T$ of \cref{prop:equiv2} (3).
\end{proof}

\begin{rmk}\label{rmk:success}
One notable success of this theorem is that, despite \cref{conj:mconj2} appearing considerably stronger than \cref{conj:mconj}, it does not require much additional effort to prove if we use the universal DNS line bundle, as long as we have extra knowledge about its exceptional locus. Additionally, the ``universal" DNS line bundle approach is useful in establishing the proof of \cref{thm:conj2prf}. While the exceptional locus of a line bundle on a family of curves may be complicated due to degeneration, we can obtain a concise description of it on the universal curve.
\end{rmk}

In the next section, we will see that all $\psi$ classes satisfy the condition of \cref{thm:conj2prf}. Therefore, as mentioned in \cref{rmk:success}, the proof of \cref{conj:mconj2} does not require much more than \cref{conj:mconj2}. This is one of the upshots of the DNS line bundle method developed in this section.

\begin{rmk}\label{rmk:bdry}
    We note that \cref{thm:kapapp} has the essentially same, but more explicit, classical proof of \cref{thm:conj2prf}.
\end{rmk}

\section{Results on DNS line bundles and lifting conjectures}\label{sec:results}

\subsection{Keel's theorem on semiampleness of \texorpdfstring{$\psi$}{TEXT} classes}\label{sec:Keel}

The following theorem is the origin of every result we will prove.

\begin{thm}\cite[Theorem 0.4]{Ke99}\label{thm:omegasemi} 
    Assume $2g-2+n>0$. Let $\pi:\overline{\mathcal{U}}_{g,n}\to \overline{\mathcal{M}}_{g,n}$ be the universal family of stable curves and $\overline{S}_1,\cdots,\overline{S}_n$ be the image of universal sections. Then $\Omega_{\pi}\left(\overline{S}_1+\cdots+\overline{S}_n\right)$ is nef, big, and exceptional locus is contained in the boundary $\partial \mathcal{M}_{g,n}$ regardless of the characteristic of the base field. Moreover, if the characteristic of the base field is positive, then this is a semiample line bundle on $\overline{\mathcal{U}}_{g,n}$.
\end{thm}

Keel's powerful theorem implies the semiampleness of every psi class.

\begin{cor}\label{cor:psisemi}
    Assume $n\ge 1$ and $2g-2+n>0$,
    \begin{enumerate}
        \item $\psi_1,\cdots, \psi_{n+1}$ are nef and big line bundles on $\overline{\mathcal{M}}_{g,n+1}$. Their exceptional locus is contained in the boundary.
        \item If $\text{Char }K>0$, then $\psi_1,\cdots, \psi_{n+1}$ are semiample line bundles on $\overline{\mathcal{M}}_{g,n+1}$.
    \end{enumerate}
\end{cor}

\begin{proof}
    (1) By \cref{thm:omegasemi} and \cref{prop:linecal}, $\psi_{n+1}$ is a nef and big line bundle whose exceptional locus is contained in the boundary. Let $\sigma_i\in S_{n+1}$ be a permutation that sends $i$ and $n+1$ to each other. This then induces $\sigma:\overline{\mathcal{M}}_{g,n+1}\to \overline{\mathcal{M}}_{g,n+1}$ which permutes the marked points. Then $\psi_i=\sigma^\ast\psi_{n+1}$. Hence, $\psi_i$ are nef and big line bundles. Also, since $\sigma_i$ sends boundary to boundary, its exceptional locus is also contained in the boundary. The proof of (2) is almost the same.
\end{proof}

\subsection{Characteristic \texorpdfstring{$p$}{TEXT} case}\label{sec:charp}
In this section, we will assume $\text{char }K=p>0$ unless otherwise stated. In this case, almost any candidate for a universal DNS line bundle is indeed a universal DNS line bundle. 

\begin{thm}\label{thm:corudns}
    Assume $2g-2+n>0$, $n\ge 1$.
    \begin{enumerate}
        \item If $n\ge 2$, $\psi_{n+1}$ is a universal DNS line bundle on $\mathcal{U}_{g,n}$. In addition, its exceptional locus is contained in the boundary.
        \item If $n=1$, then for any non-negative integers $m_1, m_2$, $m_1\psi_1+m_2\psi_2$ is a universal DNS line bundle on $\mathcal{U}_{g,1}$. Also, their exceptional locus is contained in the boundary.
    \end{enumerate}
\end{thm}

This follows from \cref{prop:linecal} (2), \cref{thm:seminds} (2) and \cref{cor:psisemi} (2). In \S \ref{sec:char0}, we will see that this is very different from the characteristic $0$ case, where almost every candidate for universal DNS line bundles is indeed not a universal DNS line bundle.

\begin{cor}\label{cor:conjpos}
    \begin{enumerate}
        \item \cref{conj:mconj} holds in positive characteristic.
        \item \cref{thm:mainthm1} holds.
    \end{enumerate}
\end{cor}

\cref{prop:equiv} and the existence of a universal DNS line bundle implies (1). Moreover, (1) implies (2) by \cref{thm:mapp}.

\begin{cor}\label{cor:conj2pos}
    \begin{enumerate}
        \item \cref{conj:mconj2} holds in positive characteristic.
        \item \cref{thm:mainthm2} holds.
    \end{enumerate}
\end{cor}

\cref{thm:conj2prf} implies (1). Then \cref{thm:mapp2} proves (2).

As mentioned earlier, almost every complete subvariety of $\rm{M}_g$ is constructed using the following two tools: (1) The existence of the Satake compactification, and (2) the Kodaira-Parshin Construction. Although the proof of \cref{conj:mconj} allows us to use the Kodaira-Parshin construction without any restrictions in positive characteristic, to the best of our knowledge, \cref{thm:mainthm1} is the best possible result if we only use these two tools. However, this result is still far from proving the conjecture $r(g)=g-2$. Therefore, developing a third general tool for constructing a complete subvariety of $\rm{M}_g$, or equivalently, a complete family of curves of genus $g$, is an essential and important task. One candidate for such construction is the subvariety of $\rm{M}_g$ suggested in \cite{FvdG04}, which uses the Prym map.

\subsection{Characteristic \texorpdfstring{$0$}{TEXT} case}\label{sec:char0}
In this section, we assume $\text{char }K=0$. Unlike the characteristic $p$ case, there are only a few possible universal DNS line bundles.

\begin{thm}\label{thm:impossible}
    Let $f:C\to S$ be a family of smooth curves of genus $g\ge 2$ with $n$ marked points $s_1,\cdots, s_n:S\to C$. Let $u:S\to \mathcal{M}_{g,n}$ and $v:C\to \mathcal{U}_{g,n}$ be the induced maps. Assume there exists a proper smooth curve $D$ with a map $i: D\to S$ such that the induced map $D\to \mathcal{M}_{g,n}$ is not a constant map but $D\to \mathcal{M}_g$ is a constant map.
    \begin{enumerate}    
        \item If $n=1$, then for any distinct $m_1, m_2\in \mathbb{Z}$, $v^\ast\left(m_1\psi_1+m_2\psi_2\right)$ is not a DNS line bundle.
        \item If $n\ge 2$, then $v^\ast \psi_{n+1}$ is not a DNS line bundle.
    \end{enumerate}
\end{thm}

The strategy of the proof is to reduce everything to the case where \cite[Lemma 3.4, Lemma 3.5]{Ke99} can be applied. The following lemma makes the reduction process possible.

\begin{lem}\label{lem:linepull}
    Let $X,Y$ be smooth varieties and $f:X \to Y$ be a finite flat morphism. If $\mathcal{L}$ is a line bundle such that $f^\ast\mathcal{L}$ is a trivial line bundle. Then there exists $m>0$ such that $\mathcal{L}^{\otimes m}$ is trivial.
\end{lem}

\begin{proof}
    It is enough to show that the kernel of the corresponding map $f^\ast:\text{Pic}(Y)\to \text{Pic}(X)$ is a torsion abelian group. Since $X$ and $Y$ are smooth varieties, their Picard groups are functorially isomorphic to $\text{CH}^1$, hence it is enough to prove the same statement for $\text{CH}^1$. This follows from the pullback-pushforward formula $f_\ast f^\ast \alpha=(\deg f)\alpha$.
\end{proof}

\begin{proof}[proof of \cref{thm:impossible}]
    Let $f': C'\to D$ be the pullback of $f$. Then $f'$ is a family of smooth curves of genus $g\ge 2$ with marked points $s_1',\cdots,s_n':D\to C'$. Let $S_1',\cdots, S_n'$ be the image of $s_1',\cdots,s_n'$ and $v':C'\to \mathcal{U}_{g,n}$ be the induced map between the family of curves. Since $D\to \mathcal{M}_g$ is a constant map, $f'$ is isotrivial. Therefore, by taking a finite cover of $D$, we may assume that $C'=C_0\times D$ and $f': C_0\times D\to D$ is the projection map. Then the marked points $s_1',\cdots,s_n'$ correspond to morphisms $r_1,\cdots, r_n:D\to C_0$. Since $D\to \mathcal{M}_{g,n}$ is not a constant map, at least one of $r_1,\cdots, r_n$ is a surjective map. Without loss of generality, assume $r_1$ is surjective.

    (1) Since $n=1$, the family of curves $f:C_0\times D\to D$ with a marked point $s_1':D\to C_0\times D$ is the pullback of the family of curves $\pi_1:C_0\times C_0\to C_0$ with a marked point given by the diagonal map $\Delta:C_0\to C_0\times C_0$ along the map $r_1:D\to C_0$. Let $u_0:C_0\to \mathcal{M}_{g,1}$ and $v_0:C_0\times C_0\to \mathcal{U}_{g,1}$ be the associated morphisms. Assume $v^\ast\left(m_1\psi_1+m_2\psi_2\right)$ is a DNS line bundle. By \cref{prop:dnspullback} (1), $\left(v'\right)^\ast\left(m_1\psi_1+m_2\psi_2\right)$ is also a DNS line bundle on $C_0\times D$. Since $r_1$ is a surjective morphism between proper smooth curves, it is a finite flat map. Hence, by \cref{prop:dnspullback} (2), $v_0^\ast\left(m_1\psi_1+m_2\psi_2\right)$ is also a DNS line bundle on $C_0\times C_0$. By \cref{prop:linecal} (1), 
    \[v_0^\ast\left(m_1\psi_1+m_2\psi_2\right)=\left(\pi_1^\ast\Omega_{C_0/K}\otimes \mathcal{O}(\Delta)\right)^{\otimes m_1}\otimes \left(\pi_2^\ast\Omega_{C_0/K}\otimes \mathcal{O}(\Delta)\right)^{\otimes m_2}. \]
    Let $\mathcal{L}_i=\pi_i^\ast\Omega_{C_0/K}\otimes \mathcal{O}(\Delta)$ for $i=1,2$. If $\mathcal{L}_1^{\otimes m_1}\otimes \mathcal{L}_2^{\otimes m_2}$ is a DNS line bundle, then there exists $d>0$ such that $\mathcal{L}_1^{\otimes dm_1}\otimes \mathcal{L}_2^{\otimes dm_2}$ is equivalent to an effective Cartier divisor $T$ on $C_0\times C_0$, whose support does not intersect with $\Delta$. Then there exists an open subset $U$ of $C_0\times C_0$, which contains $\Delta$, such that the restriction of $\mathcal{L}_1^{\otimes dm_1}\otimes \mathcal{L}_2^{\otimes dm_2}$ on $U$ is trivial. Since $U$ contains $\Delta$, it also contains its second-order neighborhood $\Delta_2$. Hence, that line bundle is also trivial on $\Delta_2$. By \cite[Lemma 3.5]{Ke99}, on $\Delta_2$, $\mathcal{L}_1^{\otimes dm_1}\otimes \mathcal{L}_2^{\otimes dm_2}\simeq \mathcal{L}_1^{\otimes d(m_1-m_2)}$. By \cite[Lemma 3.4]{Ke99}, $\mathcal{L}_1$ is not torsion, so $\mathcal{L}_1^{\otimes dm_1}\otimes \mathcal{L}_2^{\otimes dm_2}$ is not a trivial line bundle on $\Delta_2$ unless $m_1=m_2$. This is a contradiction. Hence $v^\ast\left(m_1\psi_1+m_2\psi_2\right)$ is not a DNS line bundle if $m_1\ne m_2$.

    (2) Assume $v^\ast \psi_{n+1}$ is a DNS line bundle. Then by \cref{prop:dnspullback} (1), $\left(v'\right)^\ast \psi_{n+1}$ is also a DNS line bundle. Then there exists $m>0$ and an effective Cartier divisor $T$ on $C_0\times D$ linearly equivalent to $\left(v'\right)^\ast \psi_{n+1}^{\otimes m}$ such that support of $T$ does not intersects with $S_1',\cdots, S_n'$. Consider $f':C_0\times D\to D$ and $s_1':D\to C_0\times D$ as a family of genus $g\ge 2$ curves with a marked point. Note that by the definition of DNS line bundle, if $\left(v'\right)^\ast \psi_{n+1}$ is a DNS line bundle with respect to a family of curves $f':C_0\times D\to D$ with $n$ marked points $s_1',\cdots, s_n':D\to C_0\times D$, then it is also a DNS line bundle with respect to a family of curves $f':C_0\times D\to D$ with a marked point $s_1':D\to C_0\times D$. This is the pullback of the family of curves $\pi_1:C_0\times C_0\to C_0$ with a marked point given by the diagonal map $\Delta: C_0\to C_0\times C_0$ along the map $r_1:D\to C_0$. Let $R$ be the set-theoretic union of $S_2',\cdots, S_n'$ and the support of $T$. Note that $R$ does not intersect $S_1'$. Moreover, consider the map $\text{id}_{C_0}\times r_1:C_0\times D\to C_0\times C_0$ between the family of curves. It is easy to see that $(\text{id}_{C_0}\times r_1)(R)$ is a closed subset of $C_0\times C_0$ that does not meet the diagonal. Let $U$ be the complement of $(\text{id}_{C_0}\times r_1)(R)$ in $C_0\times C_0$ and $V=(\text{id}_{C_0}\times r_1)^{-1}(U)$. Then, $U$ contains the diagonal. Moreover, $V$ and $U$ are smooth surfaces, and $r:=(\text{id}_{C_0}\times r_1)|_V$ is a finite flat map between them. Hence, we have a diagram.

    \[ \begin{tikzcd}
        V\arrow[r, hook]\arrow[d, "r"] &  C_0\times D \arrow[d, "\text{id}_{C_0}\times r_1"]\arrow[r, "\pi_2"] & D \arrow[d, "r_1"] \\
        U\arrow[r, hook ] &  C_0\times C_0 \arrow[r, "\pi_2" ] & C_0
    \end{tikzcd} \]
    
    Since
    \[ \left(v'\right)^\ast \psi_{n+1}\simeq \Omega_{C_0\times D/D}^1\left(S_1'+\cdots+S_n' \right) \]
    by \cref{prop:linecal} (1), we have
    \[ \left(v'\right)^\ast \psi_{n+1}|_V\simeq  \Omega_{C_0\times D/D}^1\left(S_1'\right)|_V\simeq r^\ast \left( \pi_2^\ast\Omega_{C_0/k}^1\otimes \mathcal{O}\left(\Delta\right)\right).  \]
    This is because $V$ is contained in the complement of $S_2',\cdots, S_n'$. However, since $\left(v'\right)^\ast \psi_{n+1}^{\otimes m}\simeq \mathcal{O}(T)$ and $V$ are contained in the complement of the support of $T$, $\left(v'\right)^\ast \psi_{n+1}^{\otimes m}|_V$ is trivial. Therefore, if we let $\mathcal{L}=\pi_2^\ast\Omega_{C_0/k}^1\otimes \mathcal{O}\left(\Delta\right)|_U$, then $r^\ast\mathcal{L}^{\otimes m}$ is a trivial line bundle. By \cref{lem:linepull}, $\mathcal{L}$ is a torsion element in the Picard group of $U$. Since $U$ contains the diagonal, it contains the second neighborhood $\Delta_2$ of the diagonal. We have
    \[ \mathcal{L}|_{\Delta_2}=\pi_2^\ast\Omega_{C_0/k}^1\otimes \mathcal{O}\left(\Delta\right)|_{\Delta_2}, \]
    but this is nontorsion by \cite[Lemma 3.4]{Ke99}. This is a contradiction. Therefore, $v^\ast \psi_{n+1}^{\otimes m}$ is not a DNS line bundle.
\end{proof}

\begin{cor}\label{cor:nodns}
    If $g\ge 3$ and $n\ge 2$, then there is no weakly universal DNS line bundle. If $g\ge 3$ and $n=1$, positive multiples of $\psi_1+\psi_2$ are the only possible weakly universal DNS line bundles.
\end{cor}

\begin{proof}
    First, we will show for every $n\ge 1$ and $g\ge 3$, there exists a complete family $f: C\to S$ that satisfies the condition of \cref{thm:impossible}. Let $X$ be a curve of genus $g$ that admits a nonconstant map $X\to E$ to an elliptic curve $E$. Let $P_1,\cdots, P_n$ be distinct points of $E$. Let $E\to E^n$ be the map defined by $x\mapsto (x+P_1,\cdots, x+P_n)$. Consider the pullback of $E\to E^n$ by $X^n\to E^n$. This is a map $t:D\to X^n$ from a complete $1$-dimensional $K$-scheme $D$. Note that the coordinates of $t(d)$ are distinct for any $d\in D$. Hence, by taking the normalization of $D$, this defines the family we wanted.

    If $n\ge 2$, this follows from \cref{thm:ndsclas} and \cref{thm:impossible}. If $n=1$, the same proposition and Theorem imply that the only possible universal DNS line bundles are positive or negative multiples of $\psi_1+\psi_2$. However, it is easy to see that negative multiples of $\psi_1+\psi_2$ are not even a DNS line bundle for a trivial family of the smooth curve with marked point $\text{Spec }K\to \mathcal{M}_{g,1}$. Hence the only possibility is positive multiples of $\psi_1+\psi_2$.
\end{proof}

One of the surprising implications of this corollary is that \cref{conj:mconj} is false if $n$ is large enough, even though it is always true in characteristic $p$.

\begin{thm}\label{thm:conjfalse}
    If $g\ge 3$ and $n>2g+2$, \cref{conj:mconj} is false.
\end{thm}

\begin{proof}
    If $n>2g+2$, then $\mathcal{M}_{g,n}$ is indeed a smooth variety (cf. \cite[Exercise I.F-8]{ACGH13}). Consider \cref{conj:mconj} when $u$ is just the identity map $\text{id}:\mathcal{M}_{g,n}\to \mathcal{M}_{g,n}$. Since $\mathcal{M}_{g,n}$ is a smooth variety, by \cref{prop:dnsconj}, \cref{conj:mconj} holds if and only if there exists a DNS line bundle on $\mathcal{U}_{g,n}$. However, by \cref{prop:dnspullback} (1), a DNS line bundle on $\mathcal{U}_{g,n}$ is automatically a universal DNS line bundle. This contradicts \cref{cor:nodns}.
\end{proof}

This reflects a significant difference between characteristic $p$ and characteristic $0$. In the characteristic $p$ case, one can universally find a new section. However, in the characteristic $0$ case, there is a family of smooth curves that do not have a potential new section. 

\begin{rmk}\label{rmk:KeelSadun} 
This phenomenon is analogous to the case of the moduli space of principally polarized abelian varieties. Let $\rm{A}_g$ be the (coarse) moduli space of $g$-dimensional principally polarized abelian varieties. In \cite{vdG99}, van der Geer proved that the dimension of a complete subvariety of $\rm{A}_g$ is $\le \frac{g(g-1)}{2}$. If $\text{char }K=p$, then there is a $\frac{g(g-1)}{2}$-dimensional complete subvariety, given by the locus of abelian varieties with $p$-rank $0$. However, Keel and Sadun \cite{KS03} proved that there is no such complete subvariety if $\text{char }K=0$ and $g\ge 3$. \cref{thm:conjfalse} can be thought of as evidence of a similar phenomenon in $\rm{M}_g$: it suggests that it is much harder to find a complete subvariety of $\rm{M}_g$ in characteristic zero.
\end{rmk}

\cref{thm:conjfalse} raises an interesting question: When does \cref{conj:mconj} hold in characteristic $0$? As mentioned above, we can modify \cref{conj:mconj} in many ways. These modifications have significant applications, including the complete subvariety problem.

For example, if $2\le n\le 2g+2$, we can still attempt to prove \cref{conj:mconj}. By \cref{prop:general} (1), it is enough to examine \cref{conj:mconj} for a finite flat covering $S\to \mathcal{M}_{g,n}$ given by a smooth variety $S$. The natural example to consider is the moduli space $\mathcal{M}_{g,n}(L)$ of smooth curves with marked points and a level structure. Hence, we need to classify DNS line bundles on $\mathcal{M}_{g,n}(L)$.

\begin{prob}\label{prob:level}
The following steps give an approach to examine \cref{conj:mconj} for $\mathcal{M}_{g,n}(L)$.
    \begin{enumerate}
        \item Calculate the Picard group of $\mathcal{M}_{g,n}(L)$ and the universal curve over it.
        \item Obtain a good description of line bundles on $\mathcal{M}_{g,n}(L)$ and its universal curve, analogous to that given by the Hodge class and psi classes on $\mathcal{M}_{g,n}$.
        \item Examine which gives rise to DNS line bundles.
    \end{enumerate}
\end{prob}

Note that if $n=0$, \cref{prob:level} (1) was done in \cite{Put12}.

Returning to the discussion, we can restrict our attention to proper varieties $S$. As one might expect, this is a strong restriction, and any argument like \cref{thm:conjfalse} seems hopeless. Also, to find complete subvarieties on $\rm{M}_{g,n}$, it is enough to consider this. However, even in this case, since we don't know the existence of a weakly universal DNS line bundle, and the obvious candidates do not work by \cref{cor:nodns}, this seems challenging in general.

Note that to find complete subvarieties of $\rm{M}_{g,n}$, we may restrict our attention to a very special family of smooth curves. For example, the family will be described in \cref{rmk:zaal}, where $S$ is a curve, is sufficient to find a complete surface in $\rm{M}_6$.

Perhaps the most important case is $n=1$, since here we have a candidate for a universal DNS line bundle. We conjecture the following.

\begin{conj}\label{conj:onlyhope}
\begin{enumerate}
    \item{Strong version:} $\psi_1+\psi_2$ is a semiample line bundle on $\overline{\mathcal{M}}_{g,2}$.
    \item{Weak version:} $\psi_1+\psi_2+\delta$ is a semiample line bundle on $\overline{\mathcal{M}}_{g,2}$ for some $\delta$ which is a linear combination of boundary divisors not including $\delta_{0, \left\{1,2\right\}}$.
\end{enumerate} 
\end{conj}

Note that the second assertion implies that $\psi_1+\psi_2$ is semiample on $\mathcal{U}_{g,1}$. This conjecture implies some meaningful conclusions on complete subvarieties of moduli space of curves.

\begin{thm}\label{thm:psisemiconcl}
    \cref{conj:onlyhope} (1) (resp. (2)) implies $\psi_1+\psi_2$ is a (resp. weakly) universal DNS line bundle on $\mathcal{U}_{g,1}$. Hence, this implies \cref{conj:mconj} and \cref{conj:mconj2} for $n=1$ and $g\ge 3$ (resp. for proper variety $S$). In particular, either of them implies that $\rm{M}_6$ contains a complete surface, and every closed point of $\rm{M}_{g,2}$ is contained in a complete subsurface.
\end{thm}

\cref{thm:psisemiconcl} is a combination of \cref{thm:mapp}, \cref{thm:mapp2}, \cref{thm:seminds}, \cref{thm:conj2prf} and \cref{cor:psisemi}.

\cref{conj:onlyhope} and \cref{thm:psisemiconcl} is also an advantage of our method. A nontrivial conclusion on the complete subvariety problem would be implied by the semi-ampleness of  a particular line bundle on moduli space of curves (this is a well-studied topic at least if $g=0$, eg.~\cite{BG21, DG22, Fak12, Fe15}. \cite{Ke99} studied this for general $g$).

\begin{rmk}\label{rmk:zaal}
    Let $f: C\to S$ be a family of smooth curves. This curve induces a family $f':C\times_S C\to C$ of smooth curves with a marked point given by the diagonal $\Delta:C\to C\times_S C$. In this case, if we let $v:C\times_S C\to \mathcal{U}_{g,1}$ be the induced map, by \cref{prop:linecal},
    \[ v^\ast\left(\psi_1+\psi_2 \right)=\pi_1^\ast\Omega_{C/S}\otimes \pi_2^\ast\Omega_{C/S}\otimes \mathcal{O}(2\Delta). \]
   Note that this is the line bundle that appeared in \cite{Za99}, where he speculated about its semiampleness in some special cases. Hence, \cref{conj:onlyhope} can be thought of as a generalization of Zaal's (implicit) conjecture. Note that it is easy to see that the semiampleness of $\pi_1^\ast\Omega_{C/S}\otimes \pi_2^\ast\Omega_{C/S}\otimes \mathcal{O}(2\Delta)$ for all such families also implies the semiampleness of $\psi_1+\psi_2$. In particular, this line bundle is the most important case of this conjecture.
\end{rmk}

The following theorem in \cite{Ke99}, with the previous remark, suggests that \cref{conj:onlyhope} is not entirely hopeless.

\begin{thm}\label{thm:keelsemi}\cite[Theorem 3.0]{Ke99}
    For a proper smooth curve $C$ of genus $g\ge 2$, $\pi_1\Omega_{C/k}^1\otimes \pi_2\Omega_{C/k}^1\otimes \mathcal{O}(2\Delta)$ is semiample on $C\times C$.
\end{thm}

We note that \cref{thm:keelsemi} holds even if the base field is of characteristic zero.

We will end this paper with a piece of good news. Even though the characteristic $0$ case of \cref{conj:mconj} and \cref{conj:mconj2} seems very hard, by using Kapranov's construction \cite{Ka93}, we can prove \cref{conj:mconj2} in any characteristic when $g=0$.

\begin{thm}\label{thm:kapapp}
    \cref{conj:mconj2} holds if $g=0$ regardless of the characteristic of the base field. In fact, $\psi_{n+1}$ is a universal DNS line bundle in this case.
\end{thm}

\begin{proof}
    We will prove \cref{prop:equiv2} (3). Let $f:C\to S$ be a family of smooth curves of genus $0$ with $n$ marked points $s_1,\cdots, s_n:S\to C$. By \cite[proof of Theorem 0.1]{Ka93}, $C$ embedds into to $\mathbb{P}^{n-2}$ bundle $\mathbb{P}\left(f_\ast \left(\Omega_{C/S}^1(S_1+\cdots+S_n) \right)  \right)$ over $S$. Let $i$ be this embedding. By \cite[Lemma 1.4]{Ka93}, in each fiber, the sections $s_1,\cdots, s_n$ are in general position, hence this projective bundle is trivial. Fix a trivialization $\mathbb{P}\left(f_\ast \left(\Omega_{C/S}^1(S_1+\cdots+S_n) \right)\right)\simeq \mathbb{P}^{n-2}_{S}$ such that $s_1,\cdots, s_n$ maps to $[1:0:\cdots:0],\cdots, [0:\cdots:0:1], [1:1:\cdots:1]$ and let $\rho:\mathbb{P}^{n-2}_{S}\to \mathbb{P}^{n-2}$ the projection. For any closed point $y\in C$ not contained in $S_1,\cdots, S_n$, choose a hyperplane section $H$ that contains $\rho(i(y))$ but does not contain any of $[1:0:\cdots:0],\cdots, [0:\cdots:0:1], [1:1:\cdots:1]$. Then the intersection of $C$ and $\rho^{-1}(H)$ is a codimension $1$ closed subscheme $S'$ of $C$ which contains $y$ and does not intersects with $S_1,\cdots, S_n$. Then by \cref{prop:equiv}, this gives a lifting as in the diagram of \cref{conj:mconj2}, and $x\in f'(S')$ by construction. Hence \cref{conj:mconj2} holds.
\end{proof}

\printbibliography

\end{document}